\documentclass[12pt, oneside]{article}
\usepackage{amsmath,amssymb,amsfonts}
\usepackage{amsthm}
\usepackage{graphicx}
\usepackage{float}
\usepackage[hidelinks]{hyperref}
\usepackage{geometry}
\usepackage{caption}
\usepackage{color}
\usepackage{fancyhdr}
\usepackage{graphicx}
\usepackage{color}
\usepackage{indentfirst}
\usepackage[T1]{fontenc}
\usepackage{enumerate}
\usepackage[caption=false,subrefformat=parens]{subfig}
\usepackage{multirow}
\usepackage{makecell}
\usepackage{booktabs}
\usepackage[section]{placeins}

\newtheorem{theorem}{Theorem}

\newtheorem{proposition}{Proposition}

\newtheorem{lemma}{Lemma}
\newtheorem{remark}{Remark}

\begin{document}

\title{Resolving the Gibbs Phenomenon in Fractional Fourier Series via Inverse Polynomial Reconstruction}
\author{Faiza Afzal\thanks{School of Mathematics, China University of Mining and Technology, Xuzhou, 221116, China; Department of Mathematics, Riphah International University Faisalabad, 38000, Pakistan.}, and Xu Xiao\thanks{Corresponding author. School of Mathematics and Statistics, Guangxi Normal University, Guilin, Guangxi 541004, China.}}

\maketitle

\begin{abstract}
The fractional Fourier series generalizes the classical Fourier series by introducing a rotation angle $\alpha$ in the time-frequency plane, but inherits the Gibbs phenomenon for piecewise smooth functions. Unlike the classical setting, the chirp modulation factor renders the fractional partial sum complex-valued, corrupting both real and imaginary components simultaneously and making direct adaptation of classical remedies insufficient. The Inverse Polynomial Reconstruction Method (IPRM) resolves the Gibbs phenomenon by enforcing that the Fourier coefficients of a Gegenbauer polynomial expansion match the given spectral data, rather than projecting the corrupted partial sum onto a polynomial basis. This paper extends the IPRM to fractional Fourier series for the first time. The fractional transformation matrix is derived and its conditioning is shown to be governed by an $\alpha$-independent Gram matrix, which reveals the dependence on the Gegenbauer parameter $\lambda$ and the polynomial degree $m$, while being entirely insensitive to the transform angle. An $L^{\infty}$ error estimate is established, guaranteeing exponential convergence for analytic functions. Numerical experiments on piecewise analytic test functions demonstrate complete elimination of the Gibbs phenomenon and 
confirm the theoretical predictions.
\end{abstract}

\section{Introduction}\label{sec:introduction}

The Fourier series is a foundational tool in signal processing and computational physics, yet it suffers from a fundamental drawback when applied to discontinuous or non-periodic functions: the Gibbs phenomenon \cite{wilbraham1848}. This effect induces persistent overshoots and oscillations near discontinuities, effectively limiting the maximum norm convergence rate to $\mathcal{O}(1)$ even for otherwise piecewise smooth functions \cite{gottlieb1997}. Consequently, these persistent oscillations introduce substantial difficulties in critical applications where precise edge is paramount. In medical image processing, for example, Gibbs ringing contaminates the boundaries of reconstructed tissues in magnetic resonance imaging (MRI), potentially leading to misdiagnosis \cite{archibald2002,dai2025}. Similarly, the effective resolution of these artifacts remains a highly active challenge in numerical partial differential equations (PDEs). Recent work on high-order shock-capturing schemes and entropy-based spectral stabilization demonstrates that Gibbs-type artifacts remain an active challenge in computational fluid dynamics \cite{liu2024,kolluru2024}.

Extensive research has been devoted to addressing the Gibbs phenomenon in classical Fourier analysis. A natural approach is to apply post-processing techniques to the corrupted partial sum. Filtering methods \cite{tadmor2002} attenuate high-frequency 
modes to suppress oscillations, but inevitably reduce resolution near smooth regions. A closely related strategy is mollification, which convolves the partial sum with a smooth kernel; Megha and Chandhini \cite{megha2025} recently proposed a 
Gegenbauer polynomial based mollifier and established pointwise spectral error bounds for this approach. A more fundamental strategy is the Gegenbauer reconstruction method \cite{gottlieb1997,gottlieb1992}, which achieves exponential convergence by re-projecting the corrupted partial sum onto an orthogonal polynomial 
basis. However, this direct approach imposes strict proportionality constraints on the polynomial degree $m$ and the Gegenbauer parameter $\lambda$ relative to the Fourier truncation order $N$ \cite{boyd2005,min2006}; failure to satisfy these 
constraints causes the reconstruction to converge to the corrupted Fourier sum rather than the true function.

To overcome these limitations, the Inverse Polynomial Reconstruction Method (IPRM) was developed \cite{shizgal2003,jung2004}. Rather than projecting the corrupted partial sum, IPRM reconstructs the function by enforcing that the Fourier coefficients of the polynomial expansion exactly match the given spectral data, thereby orthogonalizing the error with respect to the Fourier space and achieving exponential convergence without stringent parameter tuning. The method has been successfully extended to multi-dimensional scenarios and alternative spectral bases 
\cite{jung2005,adcock2012}. Recent work has further broadened the scope of spectral reconstruction methods: Li and Gelb \cite{li2025} introduced a Bayesian framework for spectral reprojection that improves robustness with respect to the Gegenbauer 
parameters and enables uncertainty quantification from noisy Fourier data, while Song et al. \cite{song2026} developed a hybrid method that combines filtering with stable extrapolation to reconstruct piecewise smooth functions from non-uniform Fourier data. Nevertheless, despite these advances, the application of IPRM to 
fractional Fourier series has not yet been explored, representing a notable gap in the literature.

Meanwhile, the fractional Fourier transform (FrFT) has emerged as a powerful generalization of classical Fourier analysis. By rotating the time-frequency plane by a fractional angle $\alpha$, the FrFT provides a unified representation particularly well-suited to chirp-type signals and non-stationary processes \cite{ozaktas2001,tao2009}. Its applications span optical signal processing \cite{niewelt2023}, synthetic aperture radar imaging \cite{sun2002}, and discrete multi-angle spectral computation \cite{huang2024}. However, much like its classical counterpart, the fractional Fourier series inevitably inherits the Gibbs phenomenon near discontinuities \cite{pei1999,zhu2014}. A fundamental complication, absent in the classical setting, is that the chirp modulation factor $e^{-\mathrm{i}\frac{1}{2}x^{2}\cot\alpha}$ \cite{ozaktas2001,almeida2002} causes Gibbs oscillations to corrupt both the real and imaginary components of the fractional partial sum simultaneously, rendering direct adaptation of classical remedies insufficient. To address this critical gap, this paper presents the first extension of the IPRM to one-dimensional fractional Fourier series. By adapting the IPRM framework to the fractional domain, we enable high-accuracy reconstruction of discontinuous signals within time-frequency adaptive representations. Specifically, our primary contributions are threefold:

\begin{enumerate}
\item We formulate the IPRM framework for fractional Fourier series, rigorously deriving the transformation matrix that maps fractional Fourier coefficients to Gegenbauer polynomial expansion coefficients.

\item We analyze the condition number of the associated linear system by establishing that it is governed by an $\alpha$-independent Gram matrix, and derive an $L^{\infty}$ error estimate that guarantees exponential convergence for analytic 
functions.

\item Through extensive experiments across various test functions, we demonstrate the method's effectiveness in completely mitigating the Gibbs phenomenon for discontinuous and piecewise-smooth functions. Furthermore, we show that our approach yields superior accuracy and robustness compared to both direct fractional Fourier approximation and the fractional Gegenbauer reconstruction method.
\end{enumerate}

The remainder of this paper is organized as follows. Section \ref{sec:preliminaries} reviews the necessary preliminaries, including the fractional Fourier series, Gegenbauer polynomials, and the classical IPRM framework. Section \ref{sec:IPRMFFrS} presents the IPRM formulation for fractional Fourier series, derives the associated transformation matrix, analyzes the conditioning of the coefficient matrix, and 
establishes an error estimate for the proposed method. Section \ref{sec:numerical} reports numerical experiments validating the theoretical results and provides a comparative analysis against existing methods. Finally, Section \ref{sec:conclusion} provides concluding remarks.

\section{Preliminaries}\label{sec:preliminaries}

\subsection{Fractional Fourier Series}\label{subsec:frft_series}

The fractional Fourier transform (FrFT) generalizes the classical Fourier transform by introducing a fractional rotation angle $\alpha$ in the time-frequency plane \cite{almeida2002,namias1980}. This section reviews the fundamental formulation of the one-dimensional fractional Fourier series and its associated Gibbs phenomenon. Given a function $f(x)$ defined on $[-\pi, \pi]$, its fractional Fourier series is defined as
\begin{equation}\label{Eq:FrF_series}
f(x) = \sum_{k=-\infty}^{\infty} c_{k,\alpha} \phi_{k,\alpha}(x),
\end{equation}
where the fractional Fourier basis functions are
\begin{equation}\label{Eq:FrF_basis}
\phi_{k,\alpha}(x) = e^{-i\frac{1}{2} x^2 \cot \alpha} e^{ikx}, \quad 0 < \alpha < \frac{\pi}{2},
\end{equation}
and the corresponding coefficients are given by
\[
c_{k,\alpha} = \frac{1}{2\pi} \int_{-\pi}^{\pi} f(x) \overline{\phi_{k,\alpha}(x)} \, dx.
\]
In practical computations, retaining only a finite number of terms in \eqref{Eq:FrF_series} yields the $N$-th partial sum
\begin{equation}\label{Eq:FrF_approx}
f_{N,\alpha}(x) = \sum_{k=-N}^{N} c_{k,\alpha} \phi_{k,\alpha}(x).
\end{equation}
While $f_{N,\alpha}(x)$ converges uniformly when $f(x)$ is globally smooth, jump discontinuities give rise to the Gibbs phenomenon \cite{pei1999,zhu2014}, limiting the convergence rate in the maximum norm to $\mathcal{O}(1)$ for piecewise smooth functions \cite{gottlieb1997}.

Unlike the classical Fourier partial sum, $f_{N,\alpha}(x)$ is complex-valued for $0 < \alpha < \pi/2$, since the chirp modulation factor $e^{-\mathrm{i}\frac{1}{2}x^2\cot\alpha}$ renders each basis function complex-valued with coupled real and imaginary components. As a consequence, a jump discontinuity in $f(x)$ simultaneously corrupts both $\operatorname{Re}[f_{N,\alpha}(x)]$ and $\operatorname{Im}[f_{N,\alpha}(x)]$, a complication absent in the classical setting $\alpha = \pi/2$. The IPRM framework developed in Section~\ref{sec:IPRMFFrS} is therefore formulated to reconstruct both components jointly from the fractional spectral data.

\subsection{Gegenbauer Polynomials}\label{subsec:gegenbauer}

As noted above, resolving the Gibbs phenomenon in $f_{N,\alpha}(x)$ requires a reconstruction basis that is free from periodic boundary constraints and admits exponential convergence for piecewise analytic functions. To this end, we apply the linear change of variables $x \mapsto x/\pi$, which maps the interval $[-\pi,\pi]$ to $[-1,1]$, and adopt Gegenbauer polynomials as the reconstruction basis on $[-1,1]$ throughout the remainder of this work.

The Gegenbauer polynomials, denoted as $C_l^{\lambda}(x)$, are a family of orthogonal polynomials defined on the interval $[-1, 1]$ with respect to the weight function $w(x) = (1-x^{2})^{\lambda-1/2}$ for $\lambda > -\frac{1}{2},\lambda\not=0$ \cite{gottlieb1997,gottlieb1992}. The orthogonality property is given by:
\[
\int_{-1}^{1}(1-x^{2})^{\lambda-\frac{1}{2}}C_k^{\lambda}(x)C_l^{\lambda}(x)\,dx=\delta_{kl}h_l^{\lambda},
\]
where $\delta_{kl}$ is the Kronecker delta and the normalization constant $h_l^{\lambda}$ is defined as:
\begin{equation}\label{Eq:h_l_lambda}
h_l^{\lambda}=\sqrt{\pi}C_l^{\lambda}(1)\frac{\Gamma(\lambda+1/2)}{\Gamma(\lambda)(l+\lambda)},
\end{equation}
with the boundary value $C_l^{\lambda}(1) = \frac{\Gamma(l+2\lambda)}{l!\Gamma(2\lambda)}$. In addition to their orthogonality properties, the Gegenbauer polynomials satisfy the three-term recurrence relation \cite{szeg1939}
\begin{equation}\label{Eq:Gegenbauer_three}
(l+1)C_{l+1}^{\lambda}(x)=2(l+\lambda)x\,C_l^{\lambda}(x)-(l+2\lambda-1)C_{l-1}^{\lambda}(x),\quad l\geq 1,
\end{equation}
with $C_0^{\lambda}(x)=1$ and $C_1^{\lambda}(x)=2\lambda x$. This relation provides an efficient means of evaluating $C_l^{\lambda}(x)$ for successive degrees and, as shown in Section \ref{subsec:analytical}, serves as the algebraic foundation for computing the fractional transformation matrix $W_{k,l,\alpha}$ without resort to numerical quadrature. Furthermore, the polynomials satisfy the symmetry relation
$C_l^{\lambda}(-x)=(-1)^l C_l^{\lambda}(x)$. Any function $f \in L^{2}[-1, 1]$ can be expanded in terms of the Gegenbauer polynomials as:
\begin{equation}\label{Eq:Gegenbauer_polynomials}
f(x)=\sum_{l=0}^{\infty}g_l C_l^{\lambda}(x),
\end{equation}
where the spectral expansion coefficients $g_l$ are determined by the inner product over the Gegenbauer space:
\begin{equation}\label{Eq:gl}
g_l=\frac{1}{h_l^{\lambda}}\int_{-1}^{1}(1-x^{2})^{\lambda-\frac{1}{2}}f(x)C_l^{\lambda}(x)dx.
\end{equation}
In practice numerical applications, $f(x)$ is approximated by truncating the 
expansion \eqref{Eq:Gegenbauer_polynomials} at degree $m$,
\begin{equation}\label{Eq:Gegenbauer_approx}
f_{m}(x)=\sum_{l=0}^{m}g_l C_l^{\lambda}(x),	
\end{equation}
For functions analytic on $[-1,1]$, this approximation converges exponentially in $m$ \cite{gottlieb1997}, which is the key property that motivates the use of Gegenbauer polynomials as the reconstruction basis in the IPRM framework.


\subsection{The Direct Gegenbauer Method}\label{subsec:direct}

To establish a comparative baseline for our proposed framework, we briefly review the direct Gegenbauer reconstruction method introduced by Gottlieb et al.\cite{gottlieb1997}. Consider the truncated classical Fourier partial sum of a piecewise smooth function $f(x)$ on the interval $[-1, 1]$, defined as:
\begin{equation}\label{Eq:F_approx}
f_N(x) = \sum_{k=-N}^N \hat{f}_k e^{ik\pi x},
\end{equation}
where $\hat{f}_k$ are the Fourier coefficients
\begin{equation}\label{Eq:F_coeff}
\hat{f}_k=\frac{1}{2}\int_{-1}^1 f(x) e^{-ik\pi x}\,dx.
\end{equation}
The direct method seeks to recover $f(x)$ from the corrupted partial sum $f_N(x)$ by finding the best Gegenbauer polynomial approximation of degree $m$ that is consistent with the available Fourier data. Based on the Gegenbauer expansion \eqref{Eq:Gegenbauer_approx}, the direct method constructs the finite approximation $f_N^m(x)=\sum_{l=0}^m \bar{g}_l^\lambda C_l^\lambda(x)$ by determining the coefficients $\bar{g}_l$ through projecting $f_N(x)$ onto the Gegenbauer polynomial space $\mathcal{G}_m=\mathrm{span}\{C_l^{\lambda}(x) \mid 0 \leq l \leq m\}$, that is, by enforcing
\[
\bar{g}_l = \langle f_N(x),\, C_l^{\lambda}(x)\rangle
= \sum_{k=-N}^{N}\hat{f}_k\,W_{k,l},
\]
where $W_{k,l}$ is the classical transformation matrix element, defined as
\begin{equation}\label{Eq:Wkl}
W_{k,l} = \frac{1}{h_l^{\lambda}}\int_{-1}^{1}(1-x^2)^{\lambda-\frac{1}{2}}e^{\mathrm{i}k\pi x}C_l^{\lambda}(x)\,dx.	
\end{equation}

The coefficients $\bar{g}_l$ are equivalent to those obtained by projecting the original function $f(x)$ first onto $\mathcal{F}_N$ and then onto $\mathcal{G}_m$. Since these two consecutive projections do not commute, the resulting approximation 
$f_N^m(x)$ inherits the Gibbs oscillations present in $f_N(x)$. More precisely, the total reconstruction error decomposes into a regularization error, arising from truncating the Gegenbauer expansion at degree $m$, and a truncation error, arising from replacing $f(x)$ by its corrupted partial sum $f_N(x)$. Consequently, this truncation error decays exponentially only if strict proportionality constraints are imposed on the polynomial degree $m$ and the parameter $\lambda$ relative to the Fourier truncation order $N$ (i.e., $\lambda = \alpha_0 N$ and $m = \beta_0 N$, where $\alpha_0, \beta_0$ are specific constants) \cite{boyd2005}.

If these constraints are not satisfied, for instance when $\lambda$ is held fixed while $m$ increases, the truncation error fails to decay, and the reconstruction converges to the corrupted partial sum rather than the true function, i.e.
\[
\lim_{m\rightarrow\infty}f_N^m(x)=f_N(x).
\]
This fundamental limitation motivates the development of the IPRM, which overcomes the non-commutativity of the two projections by instead enforcing orthogonality of the error with respect to the Fourier space, thereby eliminating the need for rigid parameter constraints.

\subsection{Inverse Polynomial Reconstruction Method}\label{subsec:iprm}
To overcome the limitations identified in Section \ref{subsec:direct}, the IPRM provides a highly flexible, parameter-independent alternative for resolving the Gibbs phenomenon. Rather than projecting the corrupted partial sum $f_N(x)$ in \eqref{Eq:F_approx} onto the Gegenbauer space, IPRM determines the unknown coefficients $g_l$ in the Gegenbauer expansion \eqref{Eq:Gegenbauer_approx} by enforcing that the Fourier coefficients of the reconstructed polynomial $f_m(x)$ exactly match the given Fourier data $\hat{f}_k$ in \eqref{Eq:F_coeff}.

Mathematically, this is achieved by defining the error
\[
E(x):=f_m(x) - f_N(x),
\]
and requiring it to be orthogonal to the Fourier space $\mathcal{F}_N$ rather than the Gegenbauer space $\mathcal{G}_m$:
\[
\langle E(x),\, e^{i k\pi x}\rangle = 0, 
\quad \forall\, k \in [-N, N].
\]
Substituting $f_m(x) = \sum_{l=0}^{m}g_l C_l^{\lambda}(x)$ and using the fact that $\langle f_N(x), e^{\mathrm{i}k\pi x}\rangle=\hat{f}_k$ (as defined in Section \ref{subsec:direct}). This orthogonality condition is equivalent to requiring that the 
Fourier coefficients of $f_m(x)$ equal the known data $\hat{f}_k$
\[
\sum_{l=0}^{m} g_l \cdot \frac{1}{2}\int_{-1}^{1} 
C_l^{\lambda}(x)\,e^{-\mathrm{i}k\pi x}\,dx = \hat{f}_k,
\quad \forall\, k \in [-N, N],
\]
which yields the linear system $\overline{W}G = \hat{F}$, where the transformation matrix $\overline{W} \in \mathbb{C}^{(2N+1)\times(m+1)}$ has elements
\begin{equation}\label{Eq:W1kl}
\overline{W}_{k,l} = \frac{1}{2}\int_{-1}^{1} C_l^{\lambda}(x)\,e^{-\mathrm{i}k\pi x}\,dx,\quad k\in[-N,N],\; l\in[0,m],
\end{equation}
$\hat{F} = [\hat{f}_{-N},\ldots,\hat{f}_N]^\top$ is the vector of known Fourier coefficients, and $G = [g_0,\ldots,g_m]^\top$ is the vector of unknown Gegenbauer coefficients. Note that $\overline{W}_{k,l}$ differs from the direct method $W_{k,l}$ in \eqref{Eq:Wkl}: the former enforces orthogonality with respect to the Fourier basis, whereas the latter involves the Gegenbauer weight function and normalization constant.

Because the error $E(x)$ is orthogonalized with respect to the Fourier basis rather than the Gegenbauer basis, IPRM effectively bypasses the non-commutativity of projections that plagues the direct method. As a result, spectral convergence is achieved without imposing the rigid proportionality constraints on $m$ and $\lambda$. The only requirement is that $m \leq 2N$, which ensures the linear system is 
not underdetermined \cite{jung2004}.

While this classical IPRM framework has proven highly successful in resolving the Gibbs phenomenon for standard Fourier series, its extension to fractional Fourier series requires non-trivial modifications. In particular, the complex-valued structure of the fractional basis functions $\phi_{k,\alpha}(x)$ necessitates a reformulation of both the orthogonality condition and the associated transformation matrix. This extension is developed in detail in Section \ref{sec:IPRMFFrS}.

\section{IPRM for One-Dimensional Fractional Fourier Series}
\label{sec:IPRMFFrS}

Building on the classical IPRM framework reviewed in Section~\ref{subsec:iprm}, we now extend the method to the fractional Fourier setting. The key distinction from the classical case is that the complex-valued structure of the fractional basis functions requires the reconstruction to simultaneously handle both the real and imaginary components of the fractional partial sum, as identified in Section~\ref{subsec:frft_series}.

\subsection{IPRM Formulation in the Fractional Fourier Setting}\label{subsec:formulation}

Following the change of variables $x \mapsto x/\pi$ introduced in Section~\ref{subsec:gegenbauer}, let $f(x)$ be a piecewise smooth function defined on $[-1,1]$. The fractional Fourier approximation $f_{N,\alpha}(x)$ and basis functions $\phi_{k,\alpha}(x)$ are given by \eqref{Eq:FrF_approx} and \eqref{Eq:FrF_basis} respectively, 
where the fractional Fourier space is
\[
\mathcal{F}_{N,\alpha}=\mathrm{span}\left\{\phi_{k,\alpha}(x):-N\leq k\leq N,\;0<\alpha<\frac{\pi}{2}\right\},	
\]
and the fractional Fourier coefficients are defined as
\[
c_{k,\alpha}=\langle f,\,\phi_{k,\alpha}\rangle=\frac{1}{2}\int_{-1}^{1} f(x)\,\overline{\phi_{k,\alpha}(x)}\,dx.
\]
The IPRM seeks an approximation $f_{m,\alpha}(x)$ of $f(x)$ in the Gegenbauer polynomial space $\mathcal{G}_m=\mathrm{span}\{C_l^{\lambda}(x)\mid 0\leq l\leq m\}$
\begin{equation}\label{Eq:FrF_Gegenbauer}
f_{m,\alpha}(x)=\sum_{l=0}^{m} \hat{g_l}\,C_l^{\lambda}(x),
\end{equation}
where the unknown coefficients $\hat{g_l}$ are to be determined. Following the IPRM principle, we define the error between the fractional Fourier projection of $f_{m,\alpha}(x)$ and the given fractional Fourier data
\[
E_{\alpha}(x):=f_{m,\alpha}^{N}(x)-f_{N,\alpha}(x),
\]
where $f_{m,\alpha}^{N}(x)=\sum_{k=-N}^{N}\langle f_{m,\alpha},\phi_{k,\alpha}\rangle\phi_{k,\alpha}(x)$ denotes the fractional Fourier projection of $f_{m,\alpha}(x)$ onto $\mathcal{F}_{N,\alpha}$. The coefficients $\hat{g}_l$ are then determined by enforcing $E_{\alpha}(x)$ to be orthogonal to $\mathcal{F}_{N,\alpha}$:
\[
\langle E_{\alpha}(x),\,\phi_{k,\alpha}(x)\rangle=0, \quad \forall\,k \in [-N,N].
\]
This orthogonality condition ensures that the reconstructed function $f_{m,\alpha}(x)$ exactly reproduces the given fractional Fourier data $c_{k,\alpha}$, thereby bypassing the non-commutativity of projections that limits the direct Gegenbauer method. Substituting \eqref{Eq:FrF_Gegenbauer} into the orthogonality condition and using the fact that $\langle f_{N,\alpha}(x),\,\phi_{k,\alpha}(x)\rangle = c_{k,\alpha}$, 
we obtain
\[
\left\langle\sum_{l=0}^{m}\hat{g_l} C_l^{\lambda}(x),\,\phi_{k,\alpha}(x)\right\rangle=c_{k,\alpha}, \quad \forall\,k\in[-N,N].
\]
Expanding the inner product and interchanging the sum and integral, this becomes
\[
\sum_{l=0}^{m}\hat{g_l}\cdot\frac{1}{2}\int_{-1}^{1}C_l^{\lambda}(x)\,\overline{\phi_{k,\alpha}(x)}\,dx=c_{k,\alpha},\quad \forall\,k\in[-N,N].
\]
We define the fractional transformation matrix $W_{\alpha}\in\mathbb{C}^{(2N+1)\times(m+1)}$ with elements
\begin{equation}\label{Eq:Wkl_alpha}
W_{k,l,\alpha}=\frac{1}{2}\int_{-1}^{1} C_l^{\lambda}(x)\,\overline{\phi_{k,\alpha}(x)}\,dx=\frac{1}{2}\int_{-1}^{1} C_l^{\lambda}(x)\,e^{\frac{\mathrm{i}}{2}x^2\cot\alpha}e^{-\mathrm{i}k\pi x}\,dx,
\end{equation}
for $k\in[-N,N]$ and $l\in[0,m]$. The orthogonality condition then takes the compact matrix form
\begin{equation}\label{Eq:linear_alpha}
W_{\alpha}\,\hat{G}=C_{\alpha},
\end{equation}
where 
\[
\hat{G}=[\hat{g}_0,\,\hat{g}_1,\,\ldots,\,\hat{g}_m]^\top \in \mathbb{C}^{m+1}
\]
is the vector of unknown Gegenbauer coefficients, and
\[
C_{\alpha}=[c_{-N,\alpha},\,c_{-N+1,\alpha},\,\ldots,\,
c_{N,\alpha}]^\top \in \mathbb{C}^{2N+1}
\]
is the vector of known fractional Fourier coefficients. Explicitly, the matrix $W_{\alpha}$ has the form
\[
W_{\alpha}= 
\begin{pmatrix}
W_{-N,0,\alpha} & W_{-N,1,\alpha} & \cdots & W_{-N,m,\alpha} \\
W_{-N+1,0,\alpha} & W_{-N+1,1,\alpha} & \cdots & W_{-N+1,m,\alpha} \\
\vdots & \vdots & \ddots & \vdots \\
W_{N,0,\alpha} & W_{N,1,\alpha} & \cdots & W_{N,m,\alpha}
\end{pmatrix}.
\]

\begin{remark}
Unlike the direct Gegenbauer transformation matrix $W_{k,l}$ in \eqref{Eq:Wkl}, which involves the weight function $(1-x^2)^{\lambda-1/2}$ and normalization constant $h_l^{\lambda}$, the elements $W_{k,l,\alpha}$ are unweighted inner products of $C_l^{\lambda}(x)$ with the fractional basis function $\overline{\phi_{k,\alpha}(x)}$. When $\alpha = \pi/2$, $W_{k,l,\alpha}$ reduces to the classical IPRM transformation matrix $\overline{W}_{k,l}$ in \eqref{Eq:W1kl}.
\end{remark}

\subsection{Analytical Structure of \texorpdfstring{$W_{k,l,\alpha}$}{}}
\label{subsec:analytical}

For $l=0$, since $C_0^{\lambda}(x)=1$, equation \eqref{Eq:Wkl_alpha} reduces to
\[
W_{k,0,\alpha}=\frac{1}{2}\int_{-1}^{1} e^{\frac{\mathrm{i}}{2}x^2\cot\alpha-\mathrm{i}k\pi x}\,dx.
\]
and let $u=x-k\pi/\cot\alpha$, we obtain
\[
W_{k,0,\alpha}=\frac{1}{2}e^{-\frac{\mathrm{i}k^2\pi^2}{2\cot\alpha}}\int_{-1-\frac{k\pi}{\cot\alpha}}^{1-\frac{k\pi}{\cot\alpha}}e^{\frac{\mathrm{i}\cot\alpha}{2}u^2}\,du.
\]
Introducing $t=\beta u$ where $\beta=\frac{(1-\mathrm{i})\sqrt{\cot\alpha}}{2}$, so that $\beta^2=-\frac{\mathrm{i}\cot\alpha}{2}$ and hence $e^{\frac{\mathrm{i}\cot\alpha}{2}u^2}=e^{-t^2}$, and expressing the result in terms of the complex error function $\mathrm{erf}(z)=\frac{2}{\sqrt{\pi}}\int_0^z e^{-t^2}\,dt$, $W_{k,0,\alpha}$ can be written as
\begin{equation}\label{Eq:Wk0_alpha_new}
W_{k,0,\alpha}=\frac{(1+\mathrm{i})\sqrt{\pi}}{4\sqrt{\cot\alpha}}\,e^{-\frac{\mathrm{i}k^2\pi^2}{2\cot\alpha}}\left[\mathrm{erf}\!\left(\beta\left(1-\frac{k\pi}{\cot\alpha}\right)\right)-\mathrm{erf}\!\left(\beta\left(-1-\frac{k\pi}{\cot\alpha}\right)\right)\right].
\end{equation}
For $l=1$, using $C_1^{\lambda}(x)=2\lambda x$ and the identity
\[
x\,e^{\frac{\mathrm{i}}{2}\cot\alpha\cdot x^2 - \mathrm{i}k\pi x}=\frac{1}{\mathrm{i}\cot\alpha}\frac{d}{dx}\left(e^{\frac{\mathrm{i}}{2}\cot\alpha\cdot x^2-\mathrm{i}k\pi x}\right)+\frac{k\pi}{\cot\alpha}e^{\frac{\mathrm{i}}{2}\cot\alpha\cdot x^2-\mathrm{i}k\pi x},
\]
Then, integrating both sides of the above equation yields
\[
W_{k,1,\alpha}=\lambda\left[\frac{1}{\mathrm{i}\cot\alpha}B_{k,\alpha}+\frac{2k\pi}{\cot\alpha}\,W_{k,0,\alpha}\right],
\]
where $B_{k,\alpha}=e^{\frac{\mathrm{i}}{2}x^2\cot\alpha - \mathrm{i}k\pi x}|_{-1}^{1}$ is the boundary term, given explicitly by
\[
B_{k,\alpha}=e^{\frac{\mathrm{i}}{2}\cot\alpha}\left[e^{-\mathrm{i}k\pi}-e^{\mathrm{i}k\pi}\right]=-2\mathrm{i}e^{\frac{\mathrm{i}}{2}\cot\alpha}\sin(k\pi)=0,
\]
where we have used $\sin(k\pi)=0$ for $k \in \mathbb{Z}$. Substituting the evaluated integral back into the definition of $W_{k,1,\alpha}$ gives
\begin{equation}\label{Eq:Wk1_alpha_new}
W_{k,1,\alpha}=\frac{2\lambda k\pi}{\cot\alpha}W_{k,0,\alpha}.	
\end{equation}
Next, for $l > 1$, we derive a recurrence for $W_{k,l+1,\alpha}$ by applying the Gegenbauer three-term recurrence \eqref{Eq:Gegenbauer_three} to the definition of $W_{k,l+1,\alpha}$
\[
(l+1)W_{k,l+1,\alpha}=(l+\lambda)\int_{-1}^{1} x\,C_l^{\lambda}(x)\,\overline{\phi_{k,\alpha}(x)}\,dx.-(l+2\lambda-1)W_{k,l-1,\alpha},
\]
which shows that each element $W_{k,l,\alpha}$ can in principle be determined recursively from the preceding two elements. However, the integral $\int_{-1}^{1}x\,C_l^\lambda(x)\,\overline{\phi_{k,\alpha}(x)}\,dx$ must itself be evaluated numerically, and the resulting quadrature errors accumulate at each recursive step, rendering the recurrence numerically unreliable for large $l$. Consequently, all elements $W_{kl,\alpha}$ for $l > 1$ are computed directly by Gauss-Legendre quadrature applied to \eqref{Eq:Wkl_alpha}, using the three-term recurrence \eqref{Eq:Gegenbauer_three} only to evaluate $C_l^{\lambda}(x)$ stably at the quadrature 
nodes. Specifically, let $\{x_q, w_q\}_{q=1}^{Q}$ denote the Gauss-Legendre nodes and weights on $[-1,1]$ with $Q$ quadrature points. Then
\[
W_{k,l,\alpha} \approx \frac{1}{2}\sum_{q=1}^{Q} w_q\,C_l^{\lambda}(x_q)\,e^{\frac{\mathrm{i}}{2}x_q^2\cot\alpha-\mathrm{i}k\pi x_q},
\]
where $C_l^{\lambda}(x_q)$ is evaluated stably at each node using the three-term recurrence \eqref{Eq:Gegenbauer_three}.

\subsection{Solvability and Conditioning of the Coefficient Matrix \texorpdfstring{$W_{\alpha}$}{}}\label{subsec:solvability}
With the elements $W_{k,l,\alpha}$ computed by Gauss-Legendre quadrature and assembled into $W_{\alpha}$, we now address the conditions under which the linear system in \eqref{Eq:linear_alpha} admits a unique solution. This system consists of $2N+1$ equations in $m+1$ unknowns, and for a unique solution to exist, the number of unknowns must not exceed the number of equations, which requires $m \leq 2N$. By the Rouché--Capelli theorem \cite{golub2013}, a unique solution exists if and only if 
\[
\mathrm{rank}(W_{\alpha})=\mathrm{rank}(W_{\alpha}\mid C_{\alpha})=m+1.
\]
When $m=2N$, the system is square and the rank requirement reduces to the nonsingularity of $W_{\alpha}$. When $m<2N$, the system is overdetermined and is solved in the least-squares sense. The conditioning of $W_\alpha$ depends on 
both the ratio $N/m$ and the Gegenbauer parameter $\lambda$; this dependence is analyzed in the subsequent content.

The nonsingularity of $W_{\alpha}$ in the square case $m=2N$ is established in the following proposition; for $m<2N$, the matrix $W_\alpha$ has full column rank by the same argument, so the least-squares solution is unique.

\begin{proposition}[Nonsingularity of $W_{\alpha}$]
For $\alpha\in(0,\pi/2),\lambda>0$, and $m=2N$, the transformation matrix $W_{\alpha}\in\mathbb{C}^{(2N+1)\times(2N+1)}$ in \eqref{Eq:linear_alpha} is nonsingular.
\end{proposition}

\begin{proof}
Suppose $W_{\alpha}\hat{G} = 0$ for some $\hat{G}=[\hat{g}_0,\ldots,\hat{g}_{2N}]^\top$. Define $p(x)=\sum_{l=0}^{2N}\hat{g}_l C_l^{\lambda}(x)$. Then for each $k\in[-N,N]$,
\[
\frac{1}{2}\int_{-1}^{1}p(x)\,\overline{\phi_{k,\alpha}(x)}\,dx=0.
\]
This states that all fractional Fourier coefficients of $p(x)$ vanish. Since $\{\phi_{k,\alpha}\}_{k=-N}^{N}$ forms a complete system on the space of polynomials of degree at most $2N$ for $\alpha \in (0,\pi/2)$ \cite{pei1999}, it follows that $p(x) \equiv 0$ on $[-1,1]$. Since the Gegenbauer polynomials $\{C_l^{\lambda}\}_{l=0}^{2N}$ are orthogonal with respect to the weight function $(1-x^2)^{\lambda-1/2}$ on $[-1,1]$ and hence linearly independent, we conclude $\hat{g}_l=0$ for all $l=0,\ldots,2N$, so $W_{\alpha}$ is nonsingular.
\end{proof}

Although $W_{\alpha}$ is theoretically nonsingular (or of full column rank when $m<2N$) for all $\alpha\in(0,\pi/2)$, the numerical solution of the linear system may be sensitive to perturbations in the data $C_{\alpha}$ when the condition number 
$\kappa(W_{\alpha})$ is large. To quantify this sensitivity, we analyze how $\kappa(W_{\alpha})$ depends on $m,\lambda$ and $\alpha$. Recall that the system $\{e^{\mathrm{i}k\pi x}\}_{k\in\mathbb{Z}}$ satisfies the orthogonality relation
\[
\int_{-1}^{1}e^{\mathrm{i}k\pi x}\,e^{-\mathrm{i}l\pi x}\,dx=2\,\delta_{kl},
\]
so that $\{e_k\}_{k\in\mathbb{Z}}$ with $e_k(x)=\frac{1}{\sqrt{2}}\,e^{\mathrm{i}k\pi x}$ forms a complete orthonormal basis of $L^2[-1,1]$. Define the chirp factor $\Phi(x)=e^{\frac{\mathrm{i}}{2}\cot\alpha\cdot x^2}$, which satisfies $|\Phi(x)|=1$ for all $x\in[-1,1]$, and let $d_l(x)=\Phi(x)\,C_l^{\lambda}(x)$. Then each entry of $W_{\alpha}$ can be written as
\[
W_{k,l,\alpha}=\frac{1}{2}\int_{-1}^{1}d_l(x)\,e^{-\mathrm{i}k\pi x}\,dx
= \frac{1}{\sqrt{2}}\,c_k(d_l),
\]
where $c_k(g_l)=\langle d_l,e_k\rangle_{L^2}$ is the $k$-th Fourier coefficient of $d_l$ with respect to the orthonormal basis $\{e_k\}$. This Fourier coefficient representation connects $W_{\alpha}^*W_{\alpha}$ to the $L^2$ inner products $\langle d_l, d_m\rangle$ via the Parseval identity. Since $\|\Phi(x)\|=1$, these inner products reduce to $\int_{-1}^1 C_l^\lambda\,C_m^\lambda\,dx$, which defines a Gram matrix $\mathrm{Gr}$ that is independent of $\alpha$. The following lemma makes this precise.

\begin{proposition}[Chirp invariance]\label{prop:chirp}
Let $\mathrm{Gr}\in\mathbb{R}^{(m+1)\times(m+1)}$ be the unweighted Gram matrix
\begin{equation}\label{Eq:Gram}
\mathrm{Gr}_{l,j}=\int_{-1}^{1}C_l^{\lambda}(x)\,C_j^{\lambda}(x)\,dx,
\end{equation}
the $\mathrm{Gr}$ is independent of $\alpha$ and let $T_N\in\mathbb{C}^{(m+1)\times(m+1)}$ be the tail matrix
\begin{equation}\label{Eq:tail}
(T_N)_{l,j}=\sum_{|k|>N}\overline{c_k(d_l)}\,c_k(d_j).
\end{equation}
Then for any finite $N$,
\begin{equation}\label{Eq:WstarW}
W_{\alpha}^{*}W_{\alpha}=\frac{1}{2}\,\mathrm{Gr}-\frac{1}{2}\,T_N,
\end{equation}
where $T_N\succeq 0$.
\end{proposition}

\begin{proof}
Direct computation gives
\[
(W_{\alpha}^{*}W_{\alpha})_{l,j}
=\sum_{k=-N}^{N}\overline{W_{k,l,\alpha}}\,W_{k,j,\alpha}
=\frac{1}{2}\sum_{k=-N}^{N}\overline{c_k(d_l)}\,c_k(d_j).
\]
By the Parseval identity \cite{rudin1974} for the complete orthonormal system $\{e_k\}$,
$\sum_{k\in\mathbb{Z}}\overline{c_k(d_l)}\,c_k(d_j)=\langle d_l,d_j\rangle_{L^2}$.
Splitting the full sum into the truncated part and the tail yields
\[
(W_{\alpha}^{*}W_{\alpha})_{l,j}=\frac{1}{2}\langle d_l,d_j\rangle_{L^2}-\frac{1}{2}\sum_{|k|>N}\overline{c_k(d_l)}\,c_k(d_j).
\]
For the first term, since $|\Phi(x)|=1$,
\[
\langle d_l,d_j\rangle_{L^2}=\int_{-1}^{1}\overline{\Phi(x)}\,C_l^{\lambda}(x)\cdot\Phi(x)\,C_j^{\lambda}(x)\,dx=\int_{-1}^{1}C_l^{\lambda}(x)\,C_j^{\lambda}(x)\,dx=\mathrm{Gr}_{l,j}.
\]
The chirp factor $\Phi$ is cancelled exactly, so $\mathrm{Gr}$ is independent of $\alpha$. For the second term, $T_N$ is a Gram matrix of the vectors $\bigl(c_k(d_l)\bigr)_{|k|>N}$, hence $T_N\succeq 0$.
\end{proof}

Proposition \ref{prop:chirp} reduces the analysis of $W_\alpha^*W_\alpha$ to the Gram matrix $\mathrm{Gr}$, which is independent of the transform angle $\alpha$. It follows that the condition number $\kappa(W_\alpha)=\sigma_{\max}(W_\alpha)
/\sigma_{\min}(W_\alpha)$ is also governed by $\mathrm{Gr}$ alone. To make this precise, we derive two-sided bounds on the extreme singular values of $W_\alpha$ in terms of the eigenvalues of $\mathrm{Gr}$.

\begin{lemma}[Singular value bounds]\label{lem:sv_bounds}
For any $N\geq 1$,
\begin{equation}\label{Eq:sv_bounds}
\frac{1}{2}\bigl[\lambda_{\min}(\mathrm{Gr})-\|T_N\|\bigr]\leq \sigma_{\min}^2(W_{\alpha})\leq \sigma^2_{\max}(W_{\alpha})\leq \frac{1}{2}\,\lambda_{\max}(\mathrm{Gr}),
\end{equation}
where $\sigma_{\min}(\cdot), \sigma_{\max}(\cdot), \lambda_{\min}(\cdot), \text{ and } \lambda_{\max}(\cdot)$ are the minimum and maximum singular values, and the minimum and maximum eigenvalues, respectively. In particular, when $m$ is chosen such that $N$ is sufficiently large relative to $m$, the tail satisfies $\|T_N\| \leq \delta\,\lambda_{\min}(\mathrm{Gr})$ for any prescribed $\delta \in (0,1)$. Consequently, the condition number is strictly bounded by
\begin{equation}\label{Eq:cond_bound}
\kappa^2(W_{\alpha})\leq\frac{1}{(1-\delta)}\cdot\frac{\lambda_{\max}(\mathrm{Gr})}{\lambda_{\min}(\mathrm{Gr})}.
\end{equation}
\end{lemma}

\begin{proof}
The upper bound follows from $T_N\succeq 0$, the \eqref{Eq:WstarW} gives
\[
W_{\alpha}^{*}W_{\alpha}=\frac{1}{2}\mathrm{Gr}-\frac{1}{2}T_N\preceq\frac{1}{2}\mathrm{Gr},	
\]
so $\sigma^2_{\max}(W_{\alpha})=\lambda_{\max}(W_{\alpha}^{*}W_{\alpha})\leq\frac{1}{2}\lambda_{\max}(\mathrm{Gr})$. The lower bound follows from the Weyl inequality \cite{horn2012},
\[
\lambda_{\min}\!\bigl(\tfrac{1}{2}\mathrm{Gr}-\tfrac{1}{2}T_N\bigr)
\geq\frac{1}{2}\lambda_{\min}(\mathrm{Gr})+\lambda_{\min}(-\frac{1}{2}T_N)=\lambda_{\min}(\mathrm{Gr})-\frac{1}{2}\lambda_{\max}(T_N)\geq \lambda_{\min}(\mathrm{Gr})-\frac{1}{2}\|T_N\|.
\]
Therefore, $\sigma^2_{\min}(W_{\alpha})=\lambda_{\min}(W_{\alpha}^{*}W_{\alpha})\geq\frac{1}{2}\bigl[\lambda_{\min}(\mathrm{Gr})-\|T_N\|\bigr]$. Since $C_l^{\lambda}$ is a polynomial on the bounded interval $[-1,1]$ and $\|\Phi\|=1$, we have $d_l\in L^2[-1,1]$. Then gives $\sum_{k\in\mathbb{Z}}|c_k(d_l)|^2=\|d_l\|_{L^2}^2<\infty$, so $\|T_N\|\leq\max_l\sum_{|k|>N}|c_k(d_l)|^2\xrightarrow{N\to\infty} 0$ as the tail of a convergent series, i.e. $\lim_{N \to \infty} \|T_N\|=0$. Since the entries of $\mathrm{Gr}$ depend only on $m$ and $\lambda$ (not on $N$), $\lambda_{\min}(\mathrm{Gr})$ is fixed for a given $m$. Therefore, for any prescribed $\delta\in(0,1)$, the condition $\|T_N\|\leq\delta\,\lambda_{\min}(\mathrm{Gr})$ is satisfied whenever $N$ is sufficiently large relative to $m$. Finally, by applying the definition of a limit and dividing the upper bound by the lower bound, the \eqref{Eq:cond_bound} is obtained.
\end{proof}

By Lemma \ref{lem:sv_bounds}, the condition number of $W_{\alpha}$ is governed by the unweighted Gram matrix $\mathrm{Gr}$, which depends only on the choice of polynomial basis $C_l^{\lambda}$. We now analyze the relationship between $\lambda_{\min}(\mathrm{Gr}),\lambda_{\max}(\mathrm{Gr})$ and the parameters $\lambda,m$.

\begin{lemma}[Lower bound on $\lambda_{\min}(\mathrm{Gr})$]\label{lem:lower}
For any $\lambda\geq\frac{1}{2}$ and any $m\geq 0$,
\begin{equation}\label{Eq:lower}
\lambda_{\min}(\mathrm{Gr})\geq\min_{0\leq l\leq m}h_l^{\lambda},
\end{equation}
where $h_l^{\lambda}$ in \eqref{Eq:h_l_lambda} is the weighted orthogonality constant of $C_l^{\lambda}$.
\end{lemma}

\begin{proof}
For any unit vector $\mathbf{c}=(c_0,\ldots,c_m)^T$ with $\|\mathbf{c}\|=1$, let $p(x)=\sum_{l=0}^{m}c_l\,C_l^{\lambda}(x)$. Then
\[
\mathbf{c}^T \mathrm{Gr}\,\mathbf{c}=\int_{-1}^{1}p(x)^2\,dx.
\]
Since $\lambda\geq\frac{1}{2}$, we have $(1-x^2)^{\lambda-1/2}\leq 1$ for all $x\in[-1,1]$, hence
\[
\int_{-1}^{1}p(x)^2\,dx\geq\int_{-1}^{1}p(x)^2(1-x^2)^{\lambda-1/2}\,dx=\sum_{l=0}^{m}c_l^2\,h_l^{\lambda}\geq\min_{0\leq l\leq m} h_l^{\lambda}.
\]
The equality utilizes the orthogonality of $C_l^{\lambda}$ with respect to the weight $(1-x^2)^{\lambda-1/2}$. The final inequality then follows directly from the assumption that the coefficient vector $\mathbf{c}$ is a unit vector (i.e., $\sum_{l=0}^{m}c_l^2=1$).
\end{proof}

The condition $\lambda \geq 1/2$ in Lemma \ref{lem:lower} arises from the inequality $(1-x^2)^{\lambda-1/2} \leq 1$ on $[-1,1]$, which holds only when $\lambda \geq 1/2$ and is essential to the proof. This restriction is not a serious limitation, as the range $\lambda \geq 1/2$ already covers the most widely used Gegenbauer families, including the Legendre polynomials ($\lambda=1/2$) and the Chebyshev polynomials of the second kind ($\lambda=1$). The lower bound is useful only when $\min_{0\leq l\leq m}h_l^{\lambda}$ remains bounded away from zero as $m$ grows. To determine for which values of $\lambda$ this holds, we analyze the monotonicity of the sequence $\{h_l^{\lambda}\}_{l\geq 0}$.

\begin{lemma}[Monotonicity of $h_l^{\lambda}$]\label{lem:mono}
For any $\lambda>0$ and any $l\geq 0$. The ratio of consecutive constants satisfies
\begin{equation}\label{Eq:ratio}
\frac{h_{l+1}^{\lambda}}{h_l^{\lambda}}=\frac{l+2\lambda}{l+1}\cdot\frac{l+\lambda}{l+1+\lambda}.
\end{equation}
Consequently:
\begin{enumerate}[(i)]
\item If $0<\lambda<1$, then $h_{l+1}^{\lambda}<h_l^{\lambda}$ for all $l\geq 0$, so $\min_l h_l^{\lambda}=h_m^{\lambda}\to 0$ as $m\to\infty$.
\item If $\lambda=1$, then $h_{l+1}^{1}=h_l^{1}$ for all $l\geq 0$, i.e.\ $h_l^{1}=\frac{\pi}{2}$ is constant.
\item If $\lambda>1$, then $h_{l+1}^{\lambda}>h_l^{\lambda}$ for all $l\geq 0$, so $\min_l h_l^{\lambda}=h_0^{\lambda}>0$.
\end{enumerate}
\end{lemma}

\begin{proof}
Identity~\eqref{Eq:ratio} follows directly from the explicit formula \eqref{Eq:h_l_lambda}. Writing the ratio as
\[
\frac{h_{l+1}^{\lambda}}{h_l^{\lambda}}
=\frac{(l+2\lambda)(l+\lambda)}{(l+1)(l+1+\lambda)}
=\frac{l^2+(3\lambda)l+2\lambda^2}{l^2+(2+\lambda)l+(1+\lambda)},
\]
one verifies that this equals~$1$ if and only if $(3\lambda)l+2\lambda^2=(2+\lambda)l+(1+\lambda)$, i.e.\ $(\lambda-1)(2l+2\lambda+1)=0$. Since $2l+2\lambda+1>0$, equality holds for all $l$ if and only if $\lambda=1$. When $\frac{1}{2}\leq\lambda<1$, the factor $(\lambda-1)<0$ gives $h_{l+1}^{\lambda}/h_l^{\lambda}<1$; when $\lambda>1$, $h_{l+1}^{\lambda}/h_l^{\lambda}>1$.
For $\lambda=1$, the constant value is $h_0^{1}=\int_{-1}^{1}(1-x^2)^{1/2}\,dx=\pi/2$.
\end{proof}

The preceding two lemmas control $\lambda_{\min}(\mathrm{Gr})$ from below. To complete the condition number estimate, we now turn to $\lambda_{\max}(\mathrm{Gr})$.

\begin{lemma}[Upper bound on $\lambda_{\max}(\mathrm{Gr})$]\label{lem:upper}
For any $\lambda>0$ and $m\geq 1$, the maximum eigenvalue of $\mathrm{Gr}$ admits the upper bound
\begin{equation}\label{Eq:upper}
\lambda_{\max}(\mathrm{Gr})\leq
\begin{cases}
C_\lambda\,m^{2\lambda-1}, & 0<\lambda<1,\\[4pt]
C_\lambda\,m\ln m,         & \lambda=1,\\[4pt]
C_\lambda\,m^{4\lambda-3}, & \lambda>1,
\end{cases}
\end{equation}
where $C_\lambda>0$ is a constant depending only on $\lambda$.
\end{lemma}

\begin{proof}
Since $\mathrm{Gr}$ is a Gram matrix, it is positive semi-definite, so all eigenvalues are non-negative. Hence
\[
\lambda_{\max}(\mathrm{Gr})\leq\mathrm{tr}(\mathrm{Gr})=\sum_{l=0}^{m}\mathrm{Gr}_{l,l}=\sum_{l=0}^{m}\|C_l^{\lambda}\|_{L^2[-1,1]}^2.
\]
By the parity $C_l^{\lambda}(-x)=(-1)^l C_l^{\lambda}(x)$, the function $[C_l^{\lambda}(x)]^2$ is even, so $\|C_l^{\lambda}\|_{L^2[-1,1]}^2=2\int_0^1 [C_l^{\lambda}(x)]^2\,dx$. Ferizovi\'c \cite[Proposition~1]{Ferizovic2022} established the following asymptotic estimates for the half-interval $L^2$-norm:
\[
\int_0^1 [C_l^{\lambda}(x)]^2\,dx \leq
\begin{cases}
c_\lambda\,l^{2\lambda-2}, & 0<\lambda<1,\\[4pt]
c\,\ln l,                  & \lambda=1,\\[4pt]
c_\lambda\,l^{4\lambda-4}, & \lambda>1,
\end{cases}
\]
where the $\lambda=1$ case uses the explicit identity $2\|C_l^{(1)}\|_{L^2[0,1]}^2=\psi(l+\tfrac{3}{2})+\gamma+\log 4$ \cite{BeltranFerizovic2020} and the bound $\psi(l+\tfrac{3}{2})\leq\ln l+1$ for $l\geq 1$, with $\psi$ the digamma function. Substituting these estimates and summing over $l$ yields the three cases. For $0<\lambda<1$, since $2\lambda-2\in(-2,0)$,
\[
\sum_{l=1}^{m}l^{2\lambda-2}\leq\int_0^m x^{2\lambda-2}\,dx + 1=\frac{m^{2\lambda-1}}{2\lambda-1}+1\leq C_\lambda\,m^{2\lambda-1}.
\]
For $\lambda=1$,
\[
\sum_{l=1}^{m}\ln l \leq m\ln m.
\]
For $\lambda>1$, since $4\lambda-4>0$,
\[
\sum_{l=1}^{m}l^{4\lambda-4}\leq\int_1^{m+1} x^{4\lambda-4}\,dx\leq C_\lambda\,m^{4\lambda-3}.
\]
In each case, combining the sum estimate with the constant contribution from $l=0$ gives \eqref{Eq:upper}.
\end{proof}

We are now in a position to state the main result, which combines the eigenvalue bounds on $\mathrm{Gr}$ (Lemmas \ref{lem:lower}--\ref{lem:upper}) with the singular value reduction (Proposition \ref{prop:chirp} and Lemma \ref{lem:sv_bounds}).

\begin{theorem}[Condition number of $W_\alpha$]\label{thm:condW}
Let $\lambda\geq\frac{1}{2}$ and let $N\gg m$ as in Lemma \ref{lem:sv_bounds}. Then,
\begin{enumerate}[\upshape(i)]
\item If $\frac{1}{2}\leq\lambda<1$, $h_l^{\lambda}$ is strictly decreasing (Lemma~\ref{lem:mono}), so $\lambda_{\min}(\mathrm{Gr})\geq h_m^{\lambda}$, where $h_m^{\lambda}\to 0$ as $m\to\infty$. By \eqref{Eq:cond_bound},
\[
\kappa^2(W_\alpha)\leq\frac{1}{1-\delta}\cdot\frac{C_{\lambda}\,m^{2\lambda-1}}{h_m^{\lambda}}.
\]
Since $h_m^{\lambda}\to 0$, the condition number deteriorates as $m$ grows.

\item If $\lambda=1$, $h_l^{1}=\frac{\pi}{2}$ for all $l$ (Lemma \ref{lem:mono}), so 
$\lambda_{\min}(\mathrm{Gr})\geq\frac{\pi}{2}$. By \eqref{Eq:cond_bound},
\[
\kappa^2(W_\alpha)\leq\frac{1}{1-\delta}\cdot\frac{C_{\lambda}\,m\ln m}{\pi/2}=\frac{2C_{\lambda}}{(1-\delta)\pi}\,m\ln m.
\]
The denominator is a positive constant independent of $m$.

\item If $\lambda>1$, $h_l^{\lambda}$ is strictly increasing (Lemma \ref{lem:mono}), so $\lambda_{\min}(\mathrm{Gr})\geq h_0^{\lambda}>0$. By \eqref{Eq:cond_bound},
\[
\kappa^2(W_\alpha)\leq\frac{1}{1-\delta}\cdot\frac{C_{\lambda}\,m^{4\lambda-3}}{h_0^{\lambda}}.
\]
The denominator is a positive constant, but the numerator grows as $m^{4\lambda-3}$ with $4\lambda-3>1$.
\end{enumerate}
In all three cases, the bounds are independent of $\alpha$.
\end{theorem}

\begin{proof}
Each case follows by substituting the corresponding upper bound on $\lambda_{\max}(\mathrm{Gr})$ (Lemma \ref{lem:upper}) and $\lambda_{\min}(\mathrm{Gr})\geq\min_l h_l^{\lambda}$ (Lemma \ref{lem:lower}) into \eqref{Eq:cond_bound}, with $\min_l h_l^{\lambda}$ determined by Lemma \ref{lem:mono}. Independence of $\alpha$ follows from the fact that $\mathrm{Gr}$ does not depend on $\alpha$ (Proposition \ref{prop:chirp}).
\end{proof}

The growth of $\kappa(W_\alpha)$ with $m$ in Theorem \ref{thm:condW} depends on the range of $\lambda$. When $\lambda<1$, the lower bound $\lambda_{\min}(\mathrm{Gr})\geq h_m^{\lambda}\to 0$ degenerates, causing $\kappa(W_\alpha)$ to deteriorate. When $\lambda\geq 1$, the lower bound remains positive, while the upper bound on $\lambda_{\max}(\mathrm{Gr})$ grows as $m\ln m$ at $\lambda=1$ and as $m^{4\lambda-3}$ for $\lambda>1$. A favorable choice of $\lambda$ must therefore balance these two effects. The numerical experiments in Section \ref{subsec:condition} identify this balance 
point at $\lambda\approx 0.75$, for which $\kappa(W_\alpha)$ grows extremely slowly with $m$. Moreover, the condition number is insensitive to the transform angle $\alpha$, consistent with the $\alpha$-independence established in Proposition \ref{prop:chirp}. Based on this analysis, we adopt $\lambda=0.75$ throughout the remainder of this paper. In applications where $\lambda$ is constrained by problem-specific considerations and cannot be freely chosen, the resulting ill-conditioning can be mitigated by standard regularization techniques such as Tikhonov regularization \cite{willoughby1979}, at the cost of introducing a regularization bias.

\subsection{Error estimate}\label{subsec:error}

We derive an $L^{\infty}$ error estimate for the fractional IPRM approximation $f_{m,\alpha}$ defined in \eqref{Eq:FrF_Gegenbauer}. The classical IPRM achieves spectral convergence for analytic functions \cite{shizgal2003,jung2004}; here we give an explicit error bound for the fractional setting by combining the conditioning analysis of $W_\alpha$ from Section \ref{subsec:solvability} with classical estimates on Gegenbauer coefficient decay. For piecewise analytic functions with finitely many discontinuities in $(-1,1)$, the method is applied to each smooth subinterval separately after an affine change of variables to $[-1,1]$.

The error arises from two sources: the truncation of the Gegenbauer expansion at degree $m$, and the mismatch between the computed coefficients $\hat{g}_l$ and the exact coefficients $g_l$ due to the finite Fourier truncation at order $N$. The 
following theorem quantifies both contributions.

\begin{theorem}[Error estimate]\label{thm:error}
Let $f$ be analytic on $[-1,1]$ and $\lambda\geq\frac{1}{2}$. Let $f_{m,\alpha}=\sum_{l=0}^{m}\hat{g}_l C_l^{(\lambda)}$ be the approximation obtained by solving $W_\alpha\hat{G}=C_\alpha$. Then there exist constants $C_1,C_2>0$ and $\sigma>0$, depending on $f$ and $\lambda$, such that
\begin{equation}\label{Eq:error}
\|f-f_{m,\alpha}\|_{L^{\infty}[-1,1]}\leq C_1\Psi_m^{(\lambda)}\,e^{-\sigma m}+C_2\Phi_m^{(\lambda)}\,\Psi_m^{(\lambda)}\sqrt{m+1}\,e^{-\sigma m},
\end{equation}
where
\[
\Psi_m^{(\lambda)}=\sum_{l=0}^{m}\frac{\bigl[C_l^{(\lambda)}(1)\bigr]^2}{h_l^{\lambda}},\qquad \Phi_m^{(\lambda)}=\biggl(\sum_{l=0}^{m}\bigl[C_l^{(\lambda)}(1)\bigr]^2\biggr)^{1/2}.
\]
\end{theorem}

\begin{proof}
By the triangle inequality,
\[
\|f-f_{m,\alpha}\|_{L^{\infty}}
\leq\underbrace{\|f-S_m f\|_{L^{\infty}}}_{E_1}
+\underbrace{\|S_m f-f_{m,\alpha}\|_{L^{\infty}}}_{E_2},
\]
where $S_m f=\sum_{l=0}^{m}g_l C_l^{(\lambda)}$ is the truncated expansion with 
coefficients $g_l$ defined in \eqref{Eq:gl}.

\medskip\noindent\textbf{Truncation error $E_1$: }For any polynomial $p^*=\sum_{l=0}^m a_l C_l^{(\lambda)}$ of degree at most $m$, the orthogonality of $C_l^{(\lambda)}$ gives $S_m p^*=p^*$. Hence
\[
f-S_m f=(f-p^*)-S_m(f-p^*).
\]
Taking the $L^{\infty}$ norm gives
\[
E_1\leq(1+\Lambda_m^{(\lambda)})
\|f-p^*\|_{L^{\infty}},
\]
where $\Lambda_m^{(\lambda)}:=\sup_{\|f\|_{L^{\infty}}\leq 1}\|S_m f\|_{L^{\infty}}$. To bound $\Lambda_m^{(\lambda)}$, note that for $\|f\|_{L^{\infty}}\leq 1$, each Gegenbauer coefficient satisfies
\[
|g_l|=\frac{1}{h_l^{\lambda}}\biggl|\int_{-1}^1 f(t)\,C_l^{(\lambda)}(t)\,(1-t^2)^{\lambda-1/2}\,dt\biggr|\leq\frac{C_l^{(\lambda)}(1)}{h_l^{\lambda}}\int_{-1}^1(1-t^2)^{\lambda-1/2}\,dt=\frac{h_0^{\lambda}\,C_l^{(\lambda)}(1)}{h_l^{\lambda}},
\]
where we used $|C_l^{(\lambda)}(t)|\leq C_l^{(\lambda)}(1)$ \cite{szeg1939} and 
$h_0^{\lambda}=\int_{-1}^1 (1-t^2)^{\lambda-1/2}\,dt$. Since $|S_m f(x)|\leq\sum_{l=0}^m|g_l|\,C_l^{(\lambda)}(1)$, it follows that
\[
\Lambda_m^{(\lambda)}\leq h_0^{\lambda}\sum_{l=0}^{m}\frac{[C_l^{(\lambda)}(1)]^2}{h_l^{\lambda}}=:h_0^{\lambda}\,\Psi_m^{(\lambda)},\qquad \Psi_m^{(\lambda)}=\sum_{l=0}^{m}\frac{\bigl[C_l^{(\lambda)}(1)\bigr]^2}{h_l^{\lambda}}.
\]
If we choose $p^*$ to be the Chebyshev projection of $f$, \cite[Theorem 8.2]{trefethen2019} gives $\|f-p^*\|_{L^{\infty}}\leq 2M\rho^{-m}/(\rho-1)$, where $\rho>1$ characterizes the analyticity region of $f$ and $M>0$ is a constant depending on $f$. Therefore
\[
E_1\leq\frac{2M(1+h_0^{\lambda}\,\Psi_m^{(\lambda)})}{\rho-1}\,\rho^{-m}.
\]
Since $\Psi_m^{(\lambda)}\geq [C_0^{(\lambda)}(1)]^2/h_0^{\lambda}=1/h_0^{\lambda}>0$, we have $1\leq h_0^{\lambda}\,\Psi_m^{(\lambda)}$, and hence $1+h_0^{\lambda}\,\Psi_m^{(\lambda)}\leq 2h_0^{\lambda}\,\Psi_m^{(\lambda)}$. Setting $\sigma=\ln\rho$ and 
$C_1=4M h_0^{\lambda}/(\rho-1)$, the truncation error becomes
\begin{equation}\label{Eq:E1}
E_1\leq C_1 \Psi_m^{(\lambda)}\,e^{-\sigma m}.
\end{equation}

\medskip\noindent\textbf{Coefficient error $E_2$: }Let $G_m=(g_0,\ldots,g_m)^T$ denote the vector of exact Gegenbauer coefficients. The exact fractional Fourier coefficients satisfy
\[
C_\alpha=W_\alpha G_m+\tau_m,
\]
where $\tau_m\in\mathbb{C}^{2N+1}$ with entries $(\tau_m)_k=\sum_{l=m+1}^{\infty}g_l\,W_{k,l,\alpha}$ accounts for the contribution of the tail $l>m$. Since $W_\alpha\hat{G}=C_\alpha$, subtraction gives
\[
W_\alpha(\hat{G}-G_m)=\tau_m,
\]
and therefore
\begin{equation}\label{Eq:coeff_err}
\|\hat{G}-G_m\|\leq\frac{\|\tau_m\|}{\sigma_{\min}(W_\alpha)}.
\end{equation}
It remains to estimate $\|\tau_m\|$. Since $|\phi_{k,\alpha}(x)|=1$ for all $x\in[-1,1]$, we have
\[
|(\tau_m)_k|=\frac{1}{2}\biggl|\int_{-1}^1(f(x)-S_m f(x))\,\overline{\phi_{k,\alpha}(x)}\,dx\biggr|\leq\frac{1}{2}\int_{-1}^1 |f(x)-S_m f(x)|\,dx\leq\|f-S_m f\|_{L^{\infty}}=E_1.
\]
Hence $\|\tau_m\|^2=\sum_{k=-N}^{N}|(\tau_m)_k|^2\leq(2N+1)\,E_1^2=(m+1)\,E_1^2$, so
\begin{equation}\label{Eq:tau_bound}
\|\tau_m\|\leq\sqrt{m+1}\,E_1\leq C_1\sqrt{m+1}\,\Psi_m^{(\lambda)}\,e^{-\sigma m}.
\end{equation}
To convert this coefficient error to an $L^{\infty}$ function error, since $|C_l^{(\lambda)}(x)|\leq C_l^{(\lambda)}(1)$,
\[
E_2=\biggl\|\sum_{l=0}^{m}(\hat{g}_l-g_l)\,C_l^{(\lambda)}\biggr\|_{L^{\infty}}\leq\sum_{l=0}^m|\hat{g}_l-g_l|\cdot\|C_{l}^{(\lambda)}\|_{L^{\infty}}=\sum_{l=0}^m|\hat{g}_l-g_l|\cdot C_{l}^{(\lambda)}(1).
\]
Applying the Cauchy--Schwarz inequality, we obtain
\[
\sum_{l=0}^m|\hat{g}_l-g_l|\cdot C_{l}^{(\lambda)}(1)\leq \left(\sum_{l=0}^m[C_{l}^{(\lambda)}(1)]^2\right)^{1/2}\cdot\left(\sum_{l=0}^m|\hat{g}_l-g_l|^2\right)^{1/2}=\Phi_m^{(\lambda)}\|\hat{G}-G_m\|,
\]
where $\Phi_m^{(\lambda)}:=\biggl(\sum_{l=0}^{m}\bigl[C_l^{(\lambda)}(1)\bigr]^2\biggr)^{1/2}$. Substituting \eqref{Eq:coeff_err} and \eqref{Eq:tau_bound} gives
\[
E_2\leq\frac{\Phi_m^{(\lambda)}}{\sigma_{\min}(W_\alpha)}\cdot C_1\sqrt{m+1}\,\Lambda_m^{(\lambda)}\,e^{-\sigma m}\leq C_2\,\Phi_m^{(\lambda)}\,\Psi_m^{(\lambda)}\sqrt{m+1}\,e^{-\sigma m},
\]
where $C_2=C_1/\sigma_{\min}(W_\alpha)$. Combining $E_1$ and $E_2$ yields \eqref{Eq:error}.
\end{proof}

Theorem \ref{thm:error} shows that the fractional IPRM achieves exponential convergence for analytic functions. The error bound \eqref{Eq:error} consists 
of the exponential factor $e^{-\sigma m}$, whose rate $\sigma$ is determined solely by the analyticity of $f$, multiplied by growth factors that depend on $m$ and $\lambda$. Among these, $\sqrt{m+1}$ is evidently polynomial; $\Phi_m^{(\lambda)}$ is polynomial since $C_l^{(\lambda)}(1)=\Gamma(l+2\lambda)/(l!\,\Gamma(2\lambda))$ grows polynomially in $l$. Since polynomial growth is always dominated by exponential decay, the overall error decays exponentially in $m$. The choice of $\lambda$ does not affect the rate $\sigma$, but controls the size of the prefactors through $\Psi_m^{(\lambda)}$ and $\Phi_m^{(\lambda)}$. 

\section{Numerical Experiments}\label{sec:numerical}

\subsection{Discussion of the condition number of \texorpdfstring{$W_{\alpha}$}{}}\label{subsec:condition}

To complement the theoretical analysis in Theorem \ref{thm:condW}, we report numerical experiments on the condition number $\kappa(W_\alpha)$ and its constituent singular values $\sigma_{\max}(W_\alpha)$ and $\sigma_{\min}(W_\alpha)$ in the overdetermined regime $N\gg m$. The square case $m=2N$ has been studied in the classical IPRM literature \cite{jung2004} and is not repeated here. Throughout, the matrix $W_\alpha$ is computed by Gauss--Legendre quadrature with $N=10m$ Fourier modes, which is large enough to make the truncation error $\|T_N\|$ negligible relative to $\lambda_{\min}(\mathrm{Gr})$, in accordance with Lemma \ref{lem:sv_bounds}.

\begin{figure}[htbp]
\centering
\subfloat[$\kappa(W_\alpha)$ versus $m$ at $\alpha=\pi/4$]{\label{fig:cond_N}\includegraphics[width=0.45\textwidth,height=4.5cm]{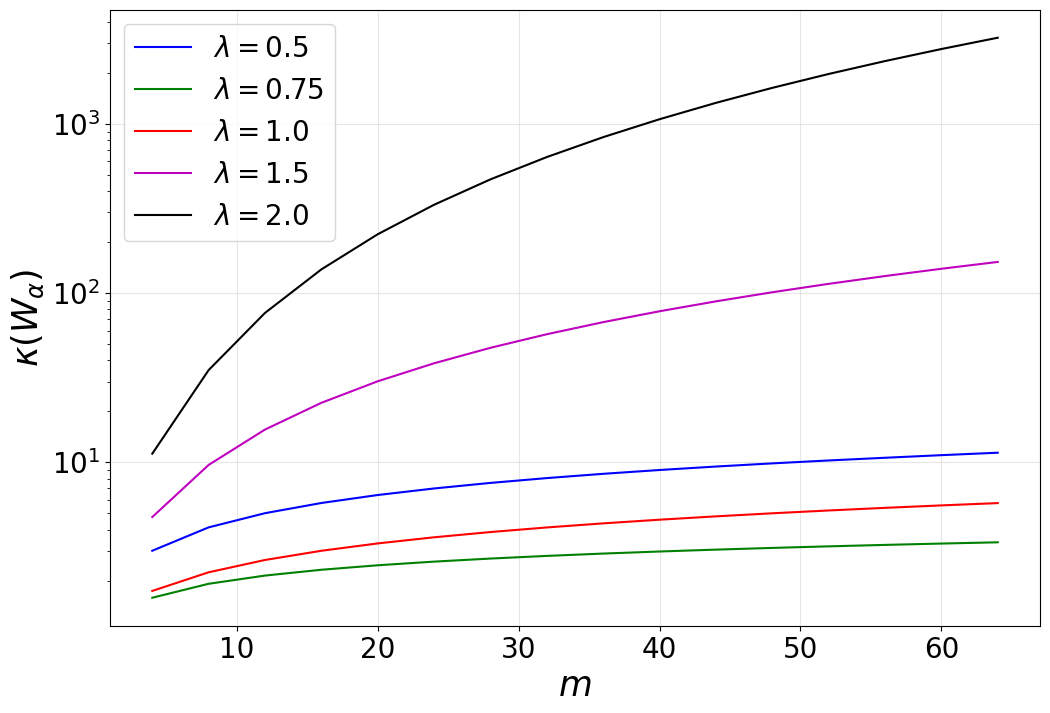}}
\hspace{1mm}
\subfloat[$\sigma_{\max}(W_\alpha)$ versus $m$ at $\alpha=\pi/4$]{\label{fig:sigma_max_N}\includegraphics[width=0.45\textwidth,height=4.5cm]{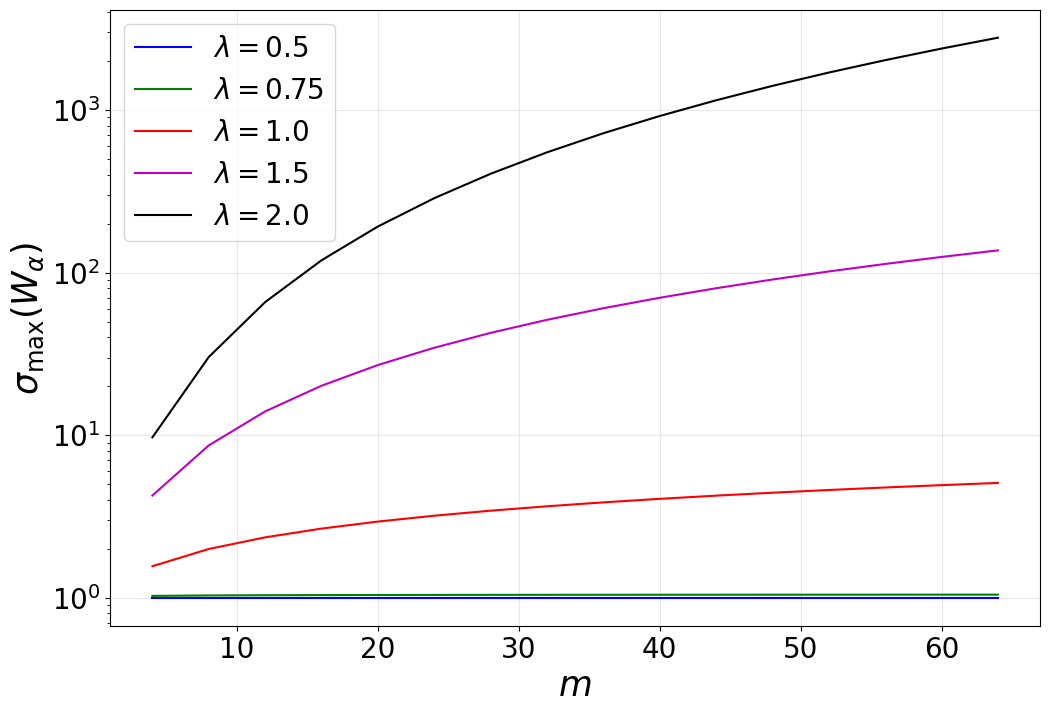}}
\hspace{1mm}
\subfloat[$\sigma_{\min}(W_\alpha)$ versus $m$ at $\alpha=\pi/4$]{\label{fig:sigma_min_N}\includegraphics[width=0.45\textwidth,height=4.5cm]{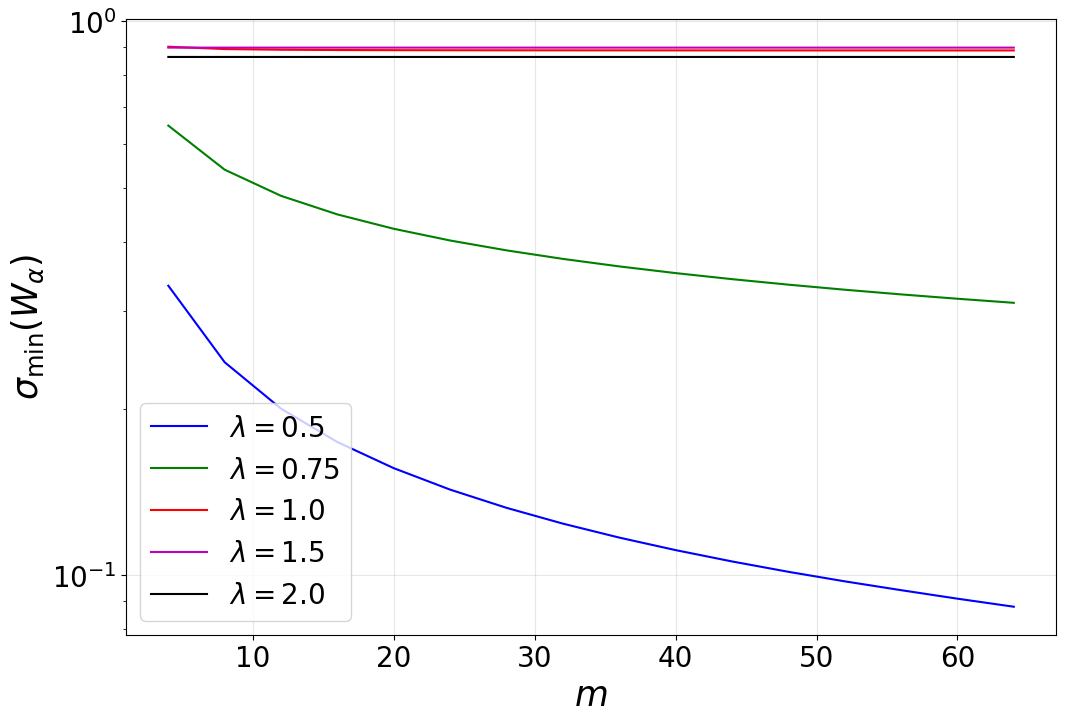}}
\hspace{1mm}
\subfloat[$\kappa(W_\alpha)$ versus $\alpha$ at $m=16$]{\label{fig:cond_alpha}\includegraphics[width=0.45\textwidth,height=4.5cm]{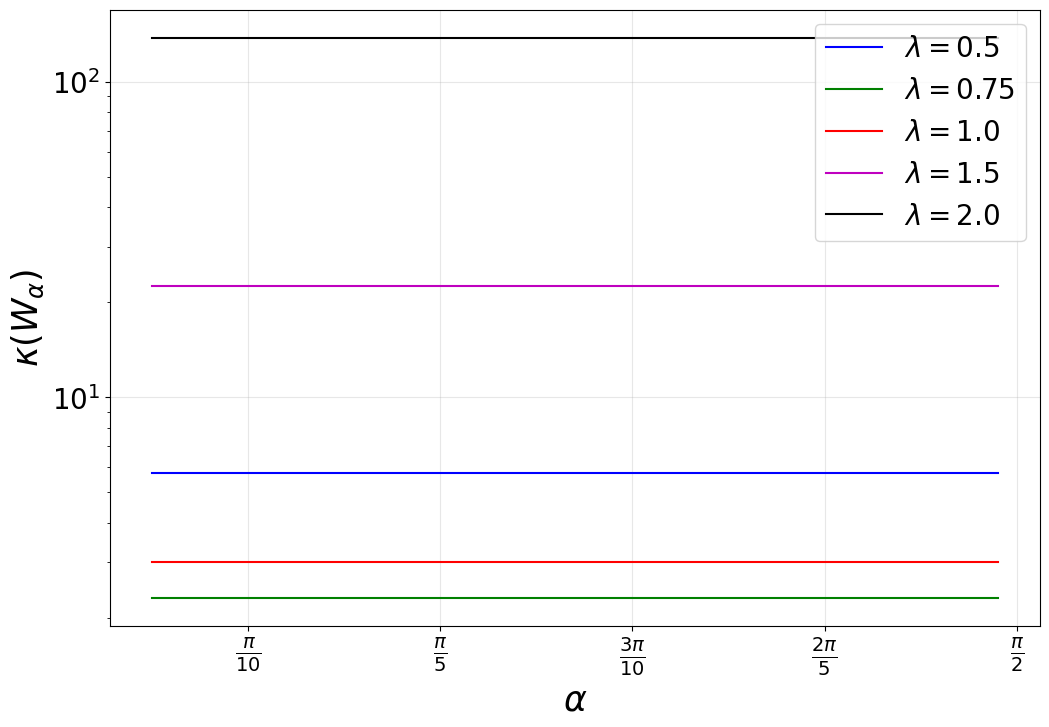}}
\hspace{1mm}
\caption{Conditioning of the coefficient matrix $W_\alpha$ for several values of the Gegenbauer parameter $\lambda$. (a) The condition number $\kappa(W_\alpha)$ as a function of the truncation level $m$. (b)--(c) The largest and smallest singular values $\sigma_{\max}(W_\alpha)$ and $\sigma_{\min}(W_\alpha)$, separating the two mechanisms that drive the growth of $\kappa(W_\alpha)$. (d) The condition number as a function of the transform angle $\alpha$, illustrating its $\alpha$-independence.}
\label{fig:cond}
\end{figure}

Figure \ref{fig:cond}\subref{fig:cond_N} shows $\kappa(W_\alpha)$ as a function of $m$ for several values of $\lambda$, with $\alpha=\pi/4$ fixed, and Figures \ref{fig:cond}\subref{fig:sigma_max_N} and \subref{fig:sigma_min_N} decompose this behavior into the contributions of $\sigma_{\max}(W_\alpha)$ and $\sigma_{\min}(W_\alpha)$ separately. The curves separate into two regimes that match the theoretical picture, with each regime driven by a different mechanism. For $\lambda<1$, the condition number grows as $m$ increases, and the decomposition shows that this growth is driven entirely by the collapse of $\sigma_{\min}(W_\alpha)$, while $\sigma_{\max}(W_\alpha)$ remains essentially flat. This reflects the degeneration of the lower bound $\lambda_{\min}(\mathrm{Gr})\geq h_m^{\lambda}\to 0$. For $\lambda\geq 1$, the situation is reversed: $\sigma_{\min}(W_\alpha)$ stays nearly constant, while $\sigma_{\max}(W_\alpha)$ grows rapidly, and the growth becomes increasingly steep as $\lambda$ increases. This reflects the growing upper bound on $\lambda_{\max}(\mathrm{Gr})$. Among the tested values, $\lambda=0.75$ sits at the crossover between these two regimes: both $\sigma_{\max}(W_\alpha)$ and $\sigma_{\min}(W_\alpha)$ vary mildly with $m$, producing the slowest overall growth of $\kappa(W_\alpha)$. This identifies $\lambda\approx 0.75$ as the empirical balance point between the two competing effects in Theorem \ref{thm:condW}.

Finally, Figure \ref{fig:cond}\subref{fig:cond_alpha} shows $\kappa(W_\alpha)$ as a function of the transform angle $\alpha\in(0,\pi/2)$ for fixed $m=16$. The curves are essentially horizontal for every value of $\lambda$, confirming that the condition number is insensitive to $\alpha$. This is consistent with the $\alpha$-independence of $\mathrm{Gr}$ established in Proposition \ref{prop:chirp}.

\subsection{Reconstruction of discontinuous functions}
\label{subsec:reconstruction}

We test the fractional IPRM on six piecewise analytic functions with jump discontinuities in $(-1,1)$, listed in Table \ref{tab:test_functions}. The functions are chosen to cover a range of features: single and multiple discontinuities, symmetric and asymmetric jumps, and varying degrees of regularity on the smooth subintervals. For each function, the discontinuity locations are assumed known, and the IPRM is applied separately on each smooth subinterval after an affine change of variables to $[-1,1]$. Throughout, we fix $\lambda=0.75$ and $N=10m$.

\begin{table}[htbp]
\centering
\caption{Test functions for the reconstruction experiments. All functions are defined on $[-1,1]$ and are analytic on each subinterval between consecutive discontinuities.}
\label{tab:test_functions}
\begin{tabular}{lcc}
\toprule
Definition & Discontinuities & Jump sizes \\
\midrule
$f_1(x):=\displaystyle\begin{cases}\dfrac{1}{1+25x^2}-1, & x<0 \\[6pt] \dfrac{1}{1+25x^2}+1, & x\geq 0 \end{cases}$ & $x=0$ & $2$ \\[10pt]
$f_2(x):=\displaystyle\begin{cases}\dfrac{1}{1+4x^2}, & x<0.3 \\[6pt] \dfrac{1}{1+4(x-0.3)^2}+1, & x\geq 0.3 \end{cases}$ & $x=0.3$ & $\approx 1.26$ \\[10pt]
$f_3(x):=\displaystyle\begin{cases}\dfrac{1}{1+16x^2}, & x<-\tfrac{1}{2} \\[6pt] \dfrac{1}{1+9x^2}+1, & -\tfrac{1}{2}\leq x<\tfrac{1}{2} \\[6pt] \dfrac{1}{1+16x^2}, & x\geq\tfrac{1}{2} \end{cases}$ & $x=\pm 1/2$ & $\approx 1.11,\,1.11$ \\[10pt]
$f_4(x):=\displaystyle\begin{cases}\tanh(10x), & x<0 \\ \tanh(10x)+2, & x\geq 0 \end{cases}$ & $x=0$ & $2$ \\[10pt]
$f_5(x):=\displaystyle\begin{cases}\tanh(6(x+\tfrac{1}{2}))-1, & x<-\tfrac{1}{2} \\ \tanh(4x)+1, & -\tfrac{1}{2}\leq x<\tfrac{1}{2} \\ \tanh(6(x-\tfrac{1}{2}))+1, & x\geq\tfrac{1}{2} \end{cases}$ & $x=\pm 1/2$ & $\approx 1.04,\,0.96$ \\[10pt]
$f_6(x):=\displaystyle\begin{cases}\tanh(8(x+\tfrac{1}{2})), & x<-\tfrac{1}{2} \\ \dfrac{1}{1+16x^2}, & -\tfrac{1}{2}\leq x<\tfrac{1}{2} \\ e^{-5(x-\frac{1}{2})}, & x\geq\tfrac{1}{2} \end{cases}$ & $x=\pm 1/2$ & $0.2,\,0.8$ \\
\bottomrule
\end{tabular}
\end{table}

Figure \ref{fig:reconstruction} displays the reconstruction results for all six test functions with $\alpha=\pi/4$ and $m=16$. The fractional Fourier partial sum $f_{N,\alpha}$ (blue dashed) exhibits pronounced Gibbs oscillations near every discontinuity. Two reconstruction methods are compared for resolving this phenomenon. The direct fractional Gegenbauer method (green dashed) reduces the oscillations but produces significant deviations from the exact function 
on the smooth subintervals, particularly for $f_1$, $f_2$, and $f_5$, where the reconstructed curve departs visibly from the true function values. This is because the direct method requires $\lambda$ to grow proportionally with $N$; when this proportionality constraint is not satisfied, the reconstruction inherits the Gibbs 
artifacts from the corrupted partial sum, as explained in Section \ref{subsec:direct}. The fractional IPRM (red dashed), by contrast, is visually indistinguishable from the exact function (black dashed) on each smooth subinterval, completely eliminating the Gibbs oscillations without introducing spurious deviations.

\begin{figure}[htbp]
\centering
\subfloat[Piecewise function: $f_1(x)$]{\label{fig:reconstruction_f1}\includegraphics[width=0.32\textwidth]{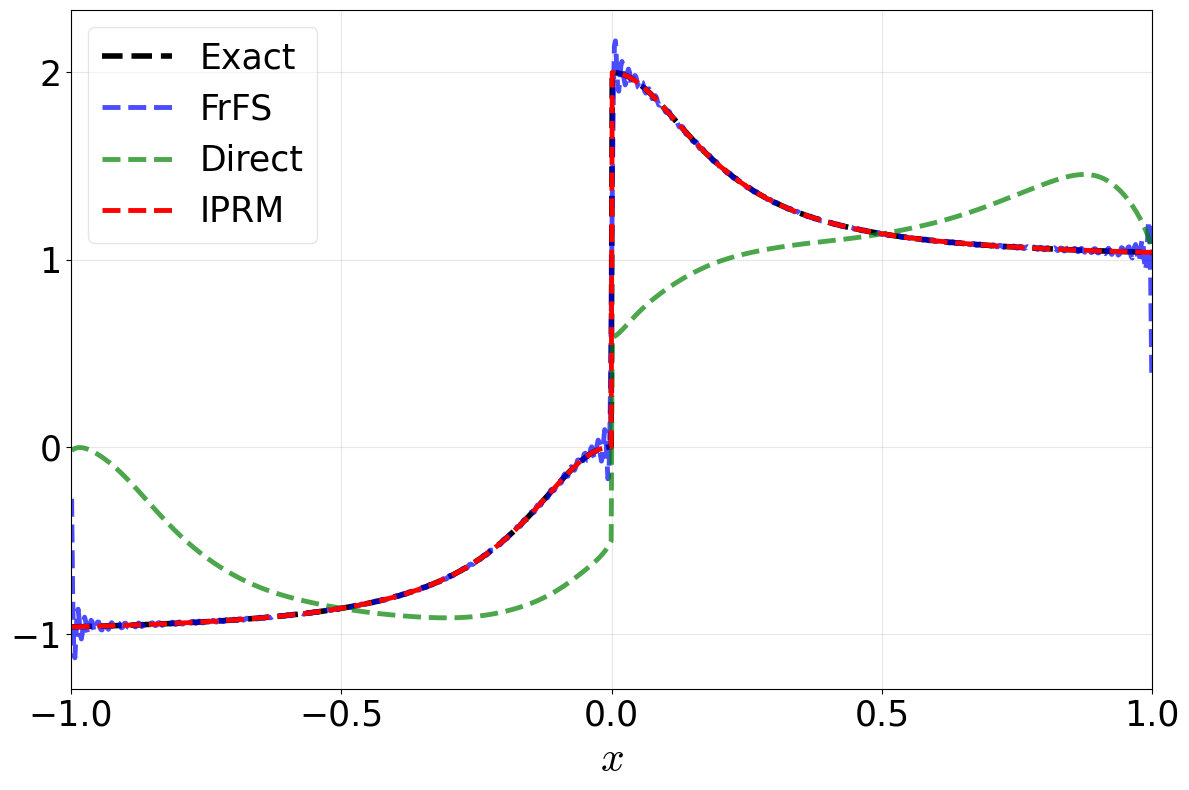}}
\hspace{1mm}
\subfloat[Asymmetric function: $f_2(x)$]{\label{fig:reconstruction_f2}\includegraphics[width=0.32\textwidth]{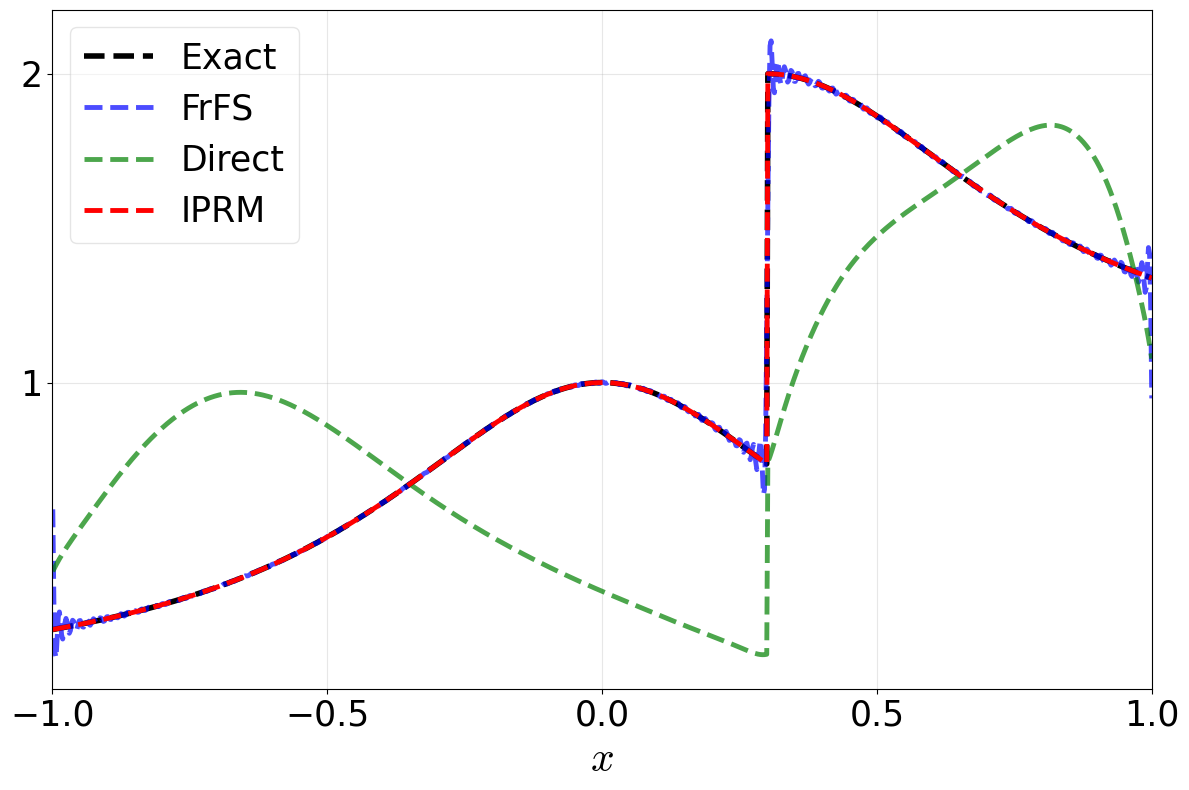}}
\hspace{1mm}
\subfloat[Two-jump function: $f_3(x)$]{\label{fig:reconstruction_f3}\includegraphics[width=0.32\textwidth]{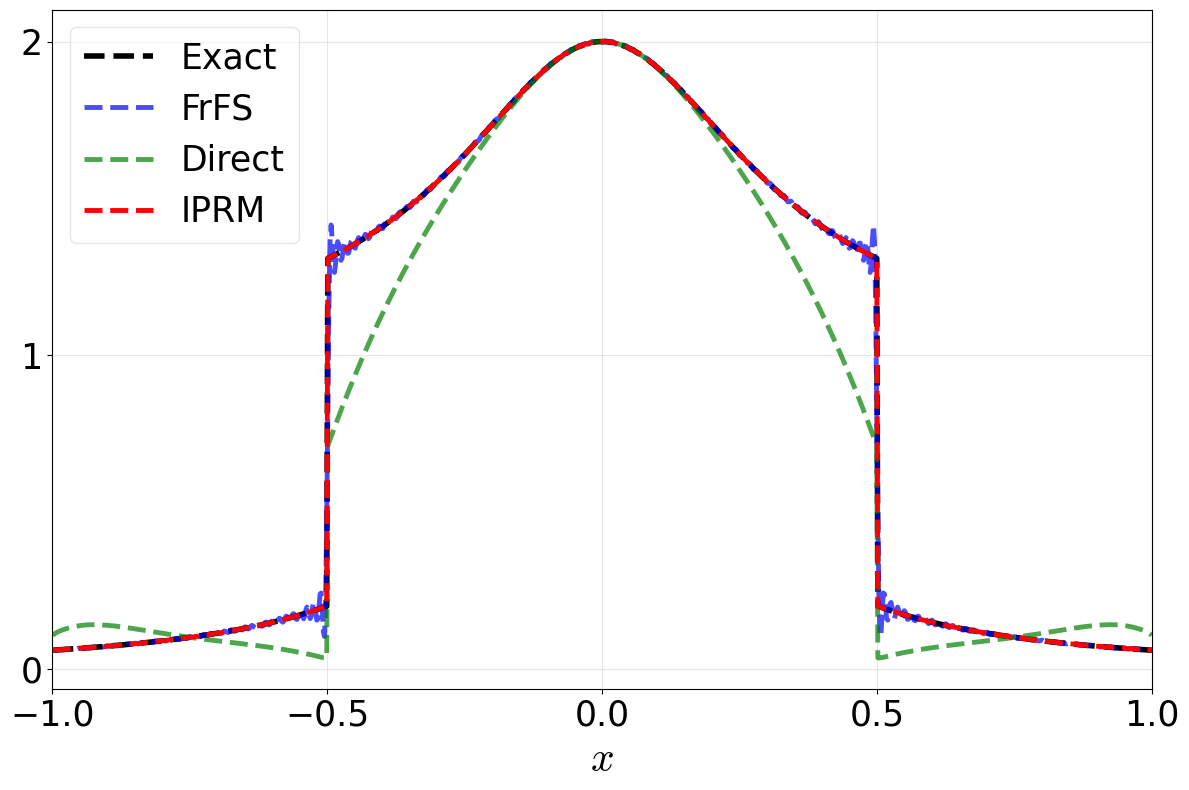}}
\hspace{1mm}
\subfloat[Piecewise tanh: $f_4(x)$]{\label{fig:reconstruction_f4}\includegraphics[width=0.32\textwidth]{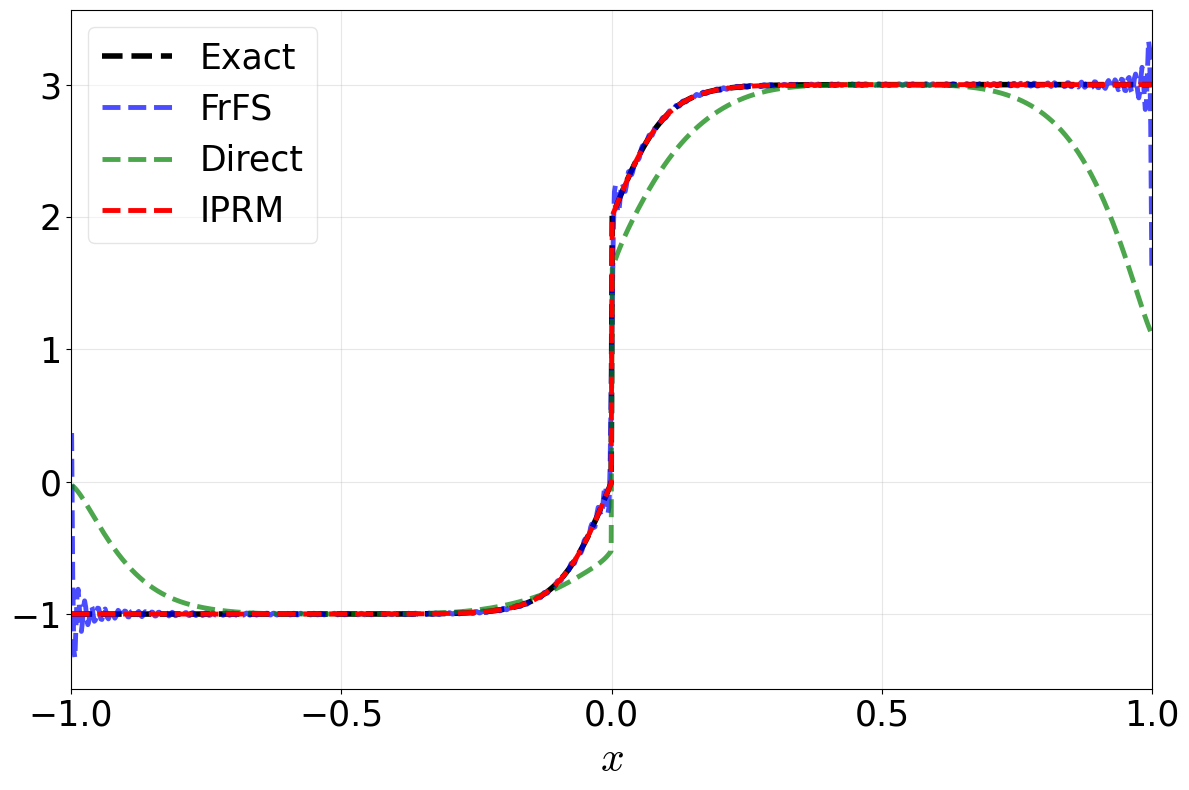}}
\hspace{1mm}
\subfloat[Two-jump tanh: $f_5(x)$]{\label{fig:reconstruction_f5}\includegraphics[width=0.32\textwidth]{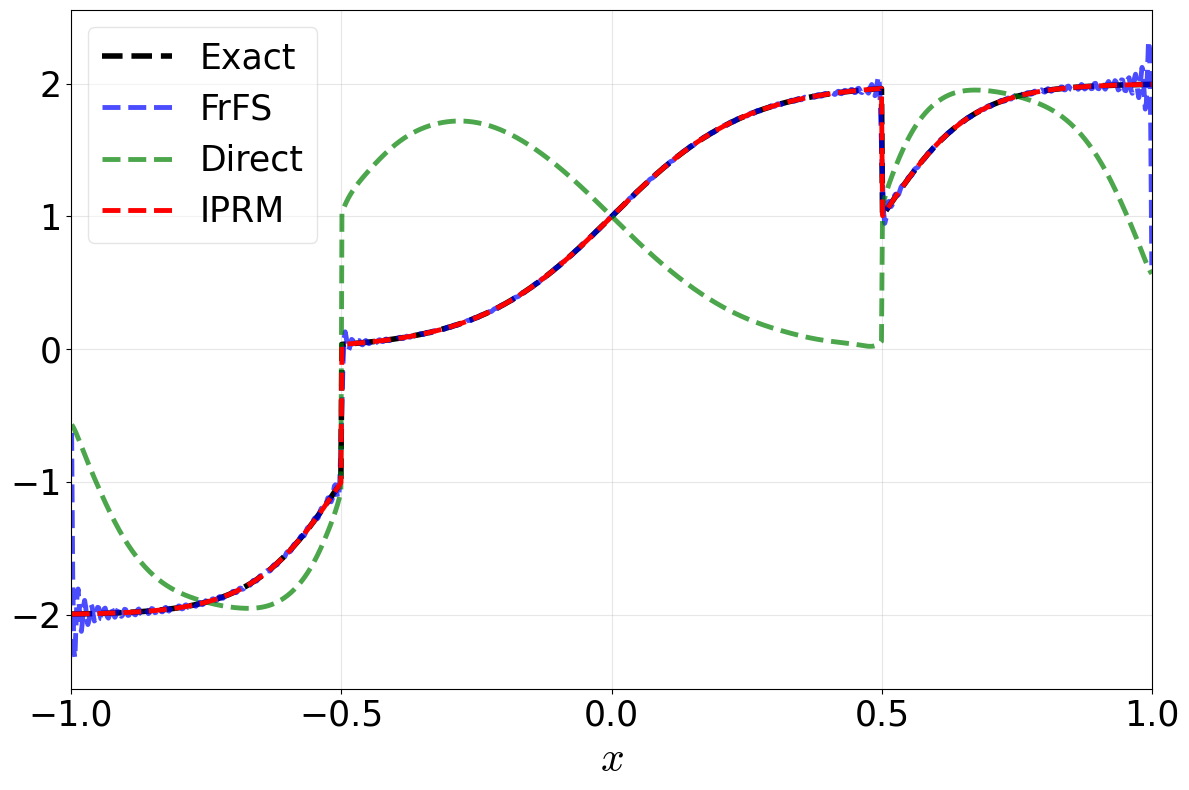}}
\hspace{1mm}
\subfloat[Three pieces: $f_6(x)$]{\label{fig:reconstruction_f6}\includegraphics[width=0.32\textwidth]{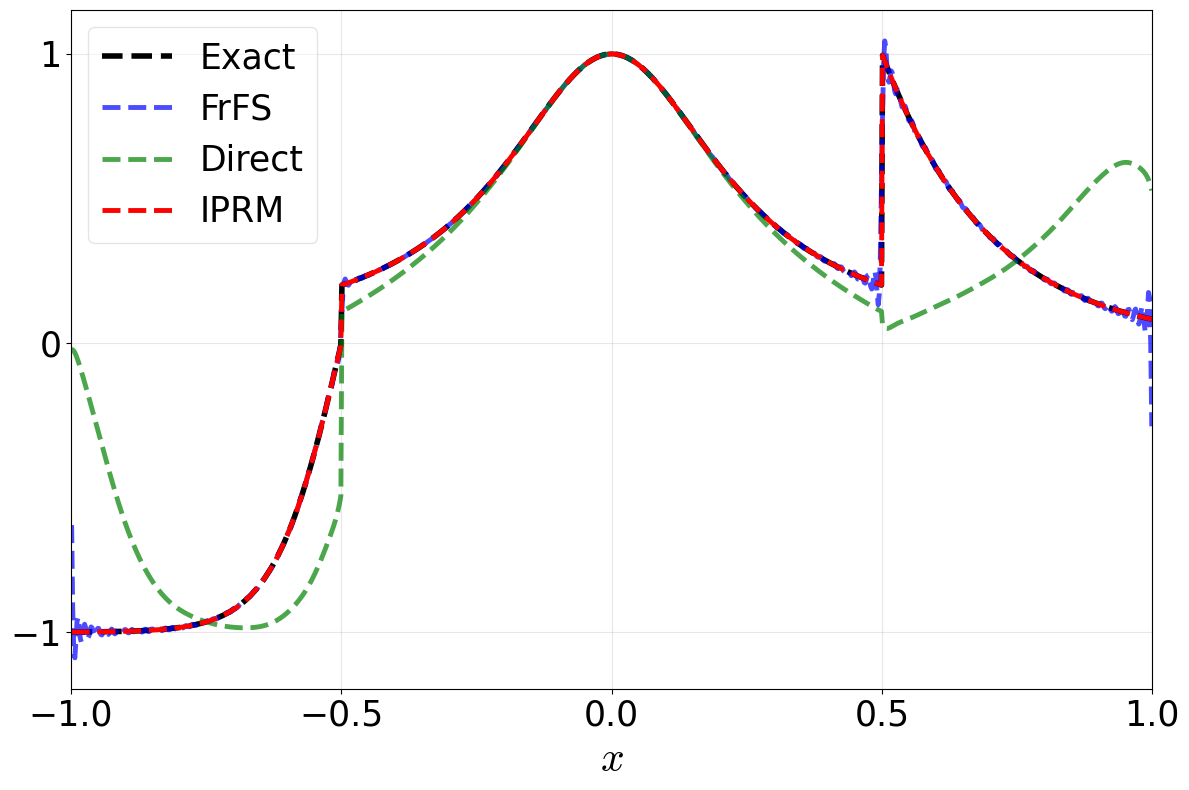}}
\caption{Reconstruction of six test functions with $\alpha=\pi/4$, $m=16$, and $\lambda=0.75$. The exact function (black dashed); fractional Fourier partial sum (blue dashed), exhibiting the Gibbs phenomenon near discontinuities; the Direct Gegenbauer method (green dashed); fractional IPRM reconstruction (red dashed), which eliminates the Gibbs oscillations and closely tracks the exact function on each smooth subinterval.}
\label{fig:reconstruction}
\end{figure}

These visual observations are confirmed quantitatively in Table \ref{tab:error_comparison}, which reports two error metrics for all three approaches at $m=16$: the relative $L^2$ error $\|f-f_{m,\alpha}\|_{L^2}/\|f\|_{L^2}$, which measures the overall reconstruction quality in an averaged sense, and the absolute $L^{\infty}$ error $\|f-f_{m,\alpha}\|_{L^{\infty}}$, which captures the worst-case pointwise deviation. For every test function, the IPRM achieves errors that are three to five orders of magnitude smaller than those of the FrFS and the direct method in both norms. The FrFS errors remain at the level of $10^{-2}$ in relative $L^2$ and $10^{-1}$ to $10^{0}$ in absolute $L^{\infty}$, reflecting the persistent 
Gibbs oscillations. The direct method performs comparably to or even worse than the FrFS, with $L^{\infty}$ errors exceeding $1$ for several functions, confirming that the reconstruction inherits the artifacts of the corrupted partial sum. By contrast, the IPRM reduces the $L^{\infty}$ error to $10^{-6}$--$10^{-4}$ across all test functions. To examine how this accuracy improves as the number of Gegenbauer coefficients increases, we next study the decay of the $L^{\infty}$ error as a function of $m$.

\begin{table}[htbp]
\caption{Relative $L^2$ and absolute $L^{\infty}$ errors of six test functions with $\alpha=\pi/4$, $m=16$ and $\lambda=0.75$. The best results are highlighted in bold.}\label{tab:error_comparison}
\centering
\setlength{\tabcolsep}{4.5mm}{
\renewcommand{\arraystretch}{1.2}
\begin{tabular}{cccc|ccc}
\toprule
\multirowcell{2}{Function} & \multicolumn{3}{c|}{Relative $L^2$ errors} & \multicolumn{3}{c}{Absolute $L^{\infty}$ errors} \\
\cline{2-7} & FrFS & Direct & IPRM & FrFS & Direct & IPRM \\
\hline
$f_1$ & 4.47e-2 & 4.46e-1 & {\bf 8.28e-6} & 6.84e-1 & 1.42e+0 & {\bf 5.38e-5}\\
$f_2$ & 2.53e-2 & 4.37e-1 & {\bf 1.25e-6} & 4.37e-1 & 1.26e+0 & {\bf 8.13e-6}\\
$f_3$ & 2.35e-2 & 1.35e-1 & {\bf 6.34e-6} & 4.77e-1 & 5.97e-1 & {\bf 3.35e-5}\\
$f_4$ & 3.49e-2 & 1.77e-1 & {\bf 1.43e-5} & 1.37e+0 & 1.89e+0 & {\bf 1.49e-4}\\
$f_5$ & 4.70e-2 & 6.58e-1 & {\bf 2.45e-6} & 1.36e+0 & 1.94e+0 & {\bf 2.21e-5}\\
$f_6$ & 3.61e-2 & 4.52e-1 & {\bf 1.33e-4} & 3.70e-1 & 9.77e-1 & {\bf 3.48e-4}\\
\bottomrule
\end{tabular}}
\end{table}

\subsection{Error decay and insensitivity to \texorpdfstring{$\alpha$}{}}\label{subsec:error_decay}

To verify the exponential convergence established in Theorem \ref{thm:error}, Figure \ref{fig:error_decay} shows the $L^{\infty}$ error of the IPRM as a function of the Gegenbauer truncation degree $m$, with $\alpha=\pi/4$ and $N=10m$. On the semi-logarithmic scale, all six test functions display approximately linear decay, consistent with the theoretical bound $e^{-\sigma m}$ where $\sigma=\ln\rho$. The slope of each curve reflects the Bernstein ellipse parameter $\rho$ of the corresponding function: a smaller $\rho$ (singularities closer to $[-1,1]$) produces a gentler slope. Among the six functions, $f_2$ converges fastest and $f_4$, $f_6$ converge slowest, in agreement with their respective $\rho$ values listed in Table \ref{tab:error_decay}.

\begin{figure}[htbp]
\centering
\subfloat[Piecewise function: $f_1(x)$]{\label{fig:error_decay_f1}\includegraphics[width=0.32\textwidth]{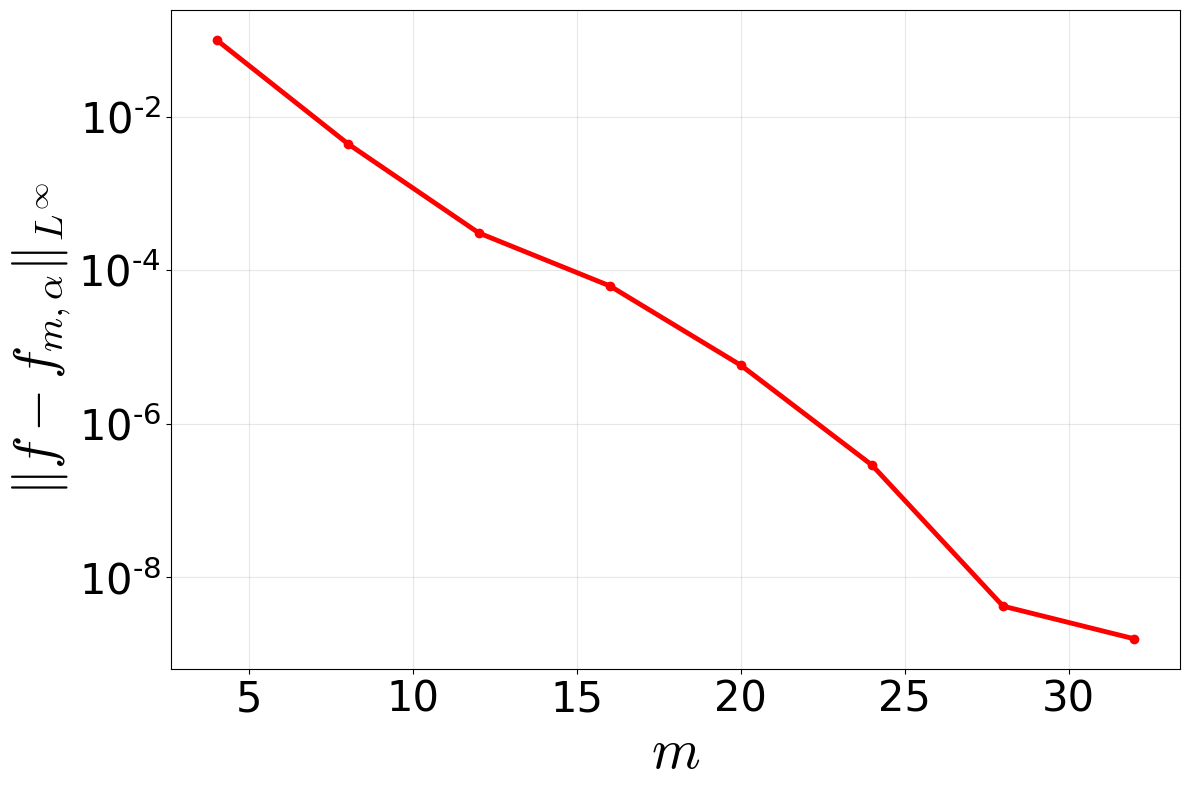}}
\hspace{1mm}
\subfloat[Asymmetric function: $f_2(x)$]{\label{fig:error_decay_f2}\includegraphics[width=0.32\textwidth]{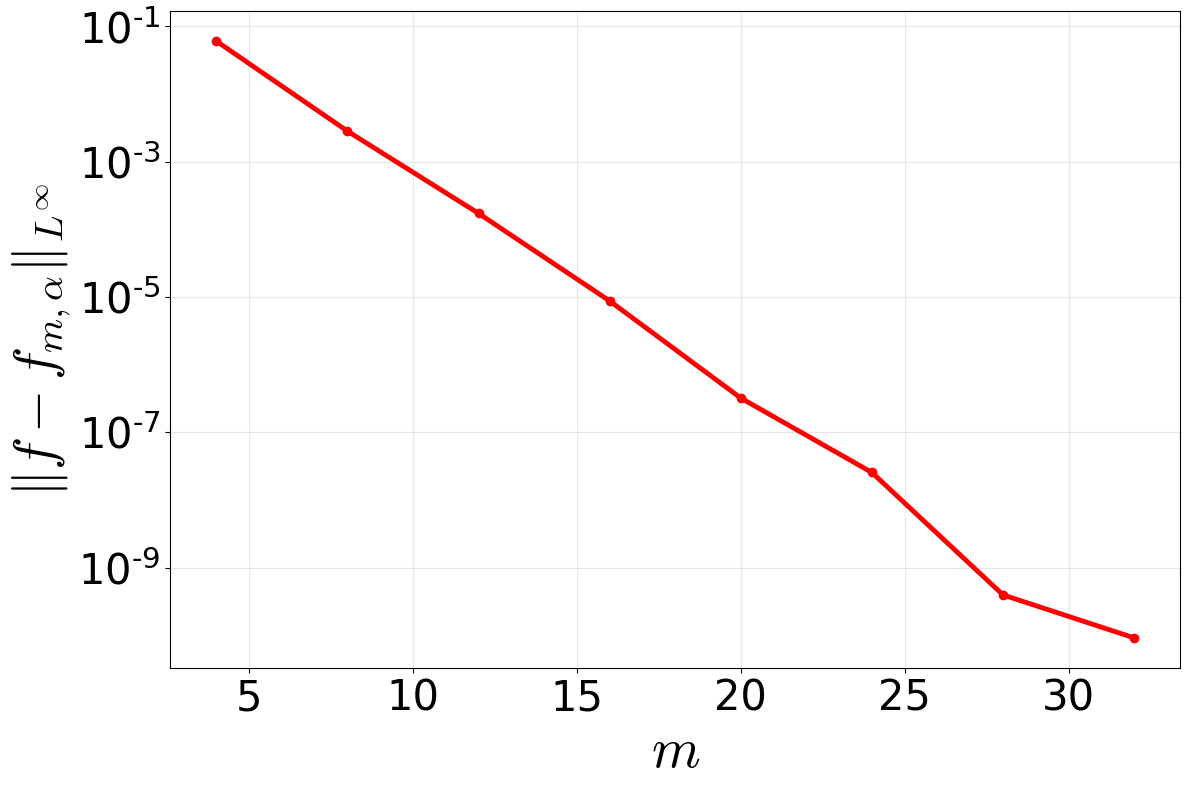}}
\hspace{1mm}
\subfloat[Two-jump function: $f_3(x)$]{\label{fig:error_decay_f3}\includegraphics[width=0.32\textwidth]{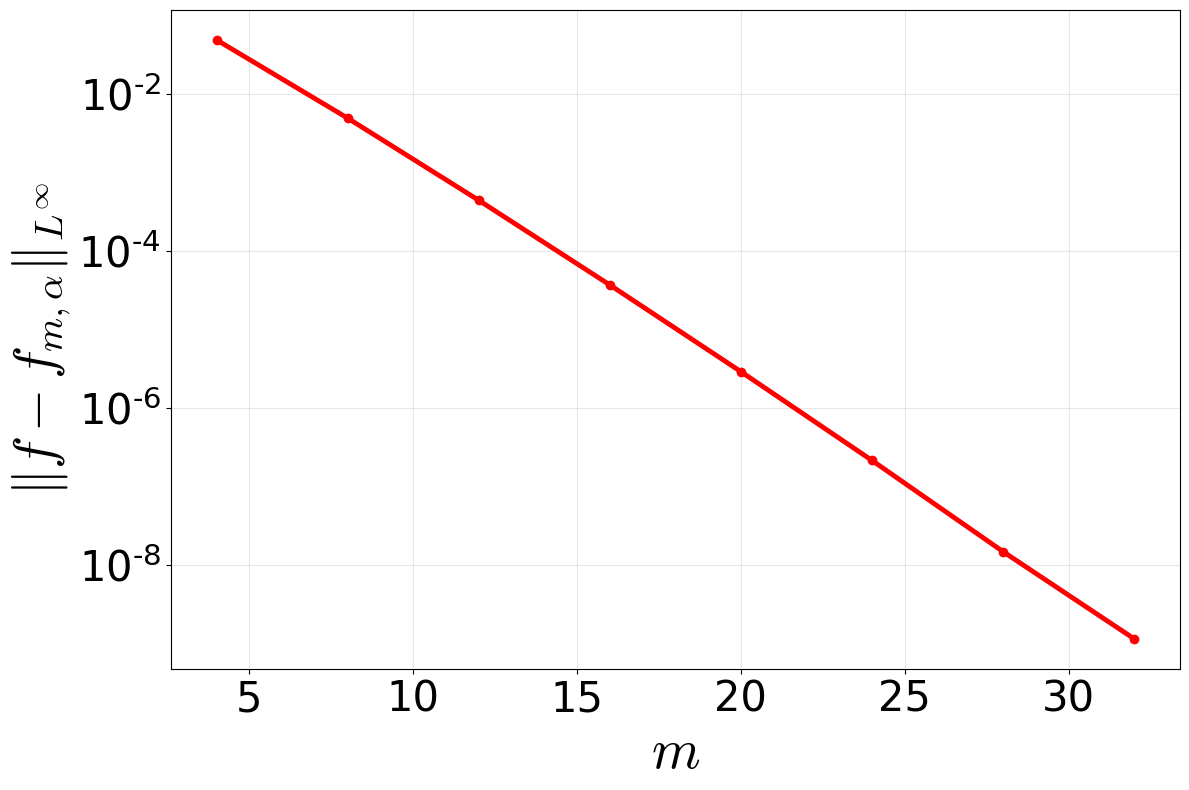}}
\hspace{1mm}
\subfloat[Piecewise tanh: $f_4(x)$]{\label{fig:error_decay_f4}\includegraphics[width=0.32\textwidth]{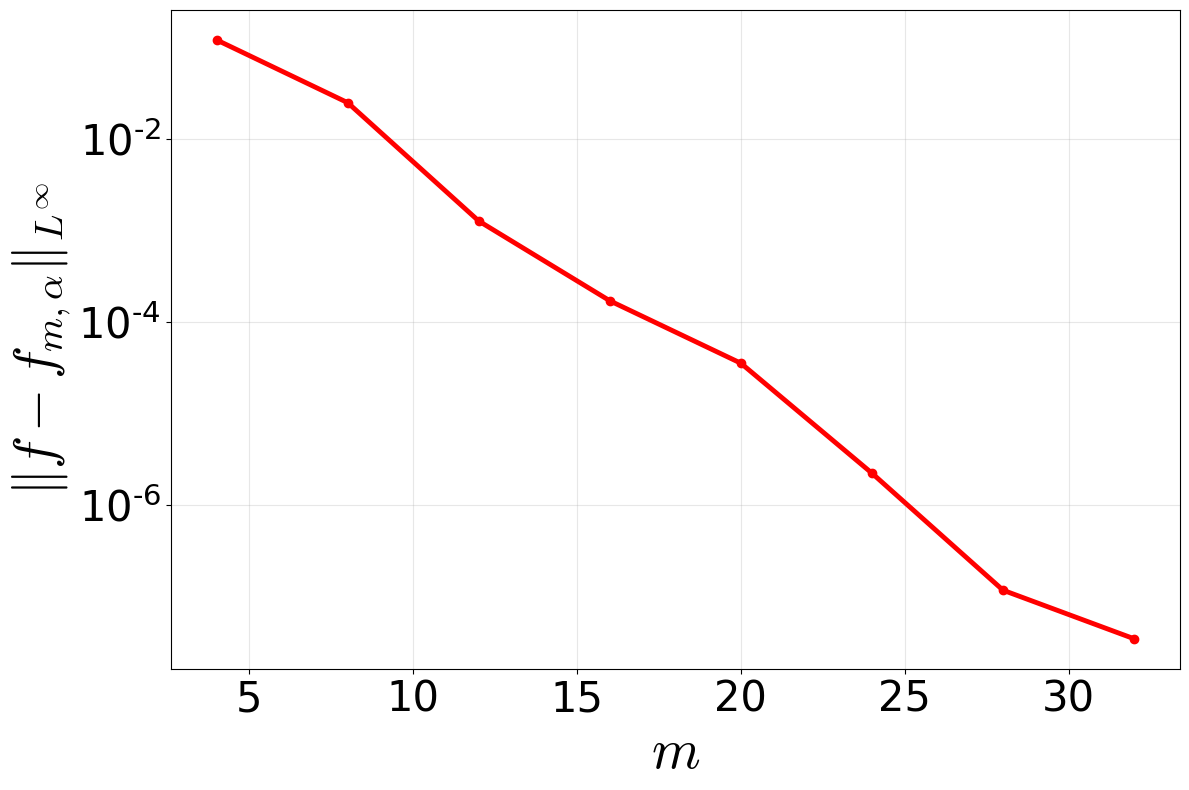}}
\hspace{1mm}
\subfloat[Two-jump tanh: $f_5(x)$]{\label{fig:error_decay_f5}\includegraphics[width=0.32\textwidth]{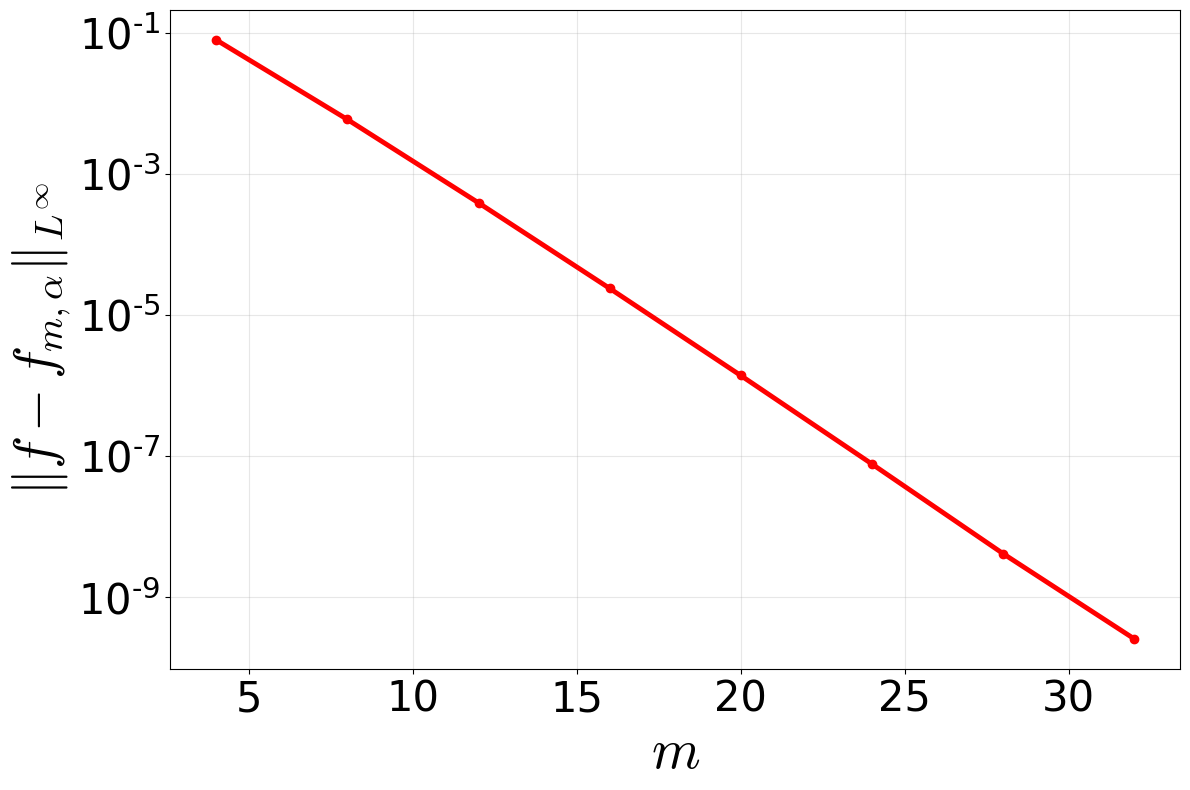}}
\hspace{1mm}
\subfloat[Three pieces: $f_6(x)$]{\label{fig:error_decay_f6}\includegraphics[width=0.32\textwidth]{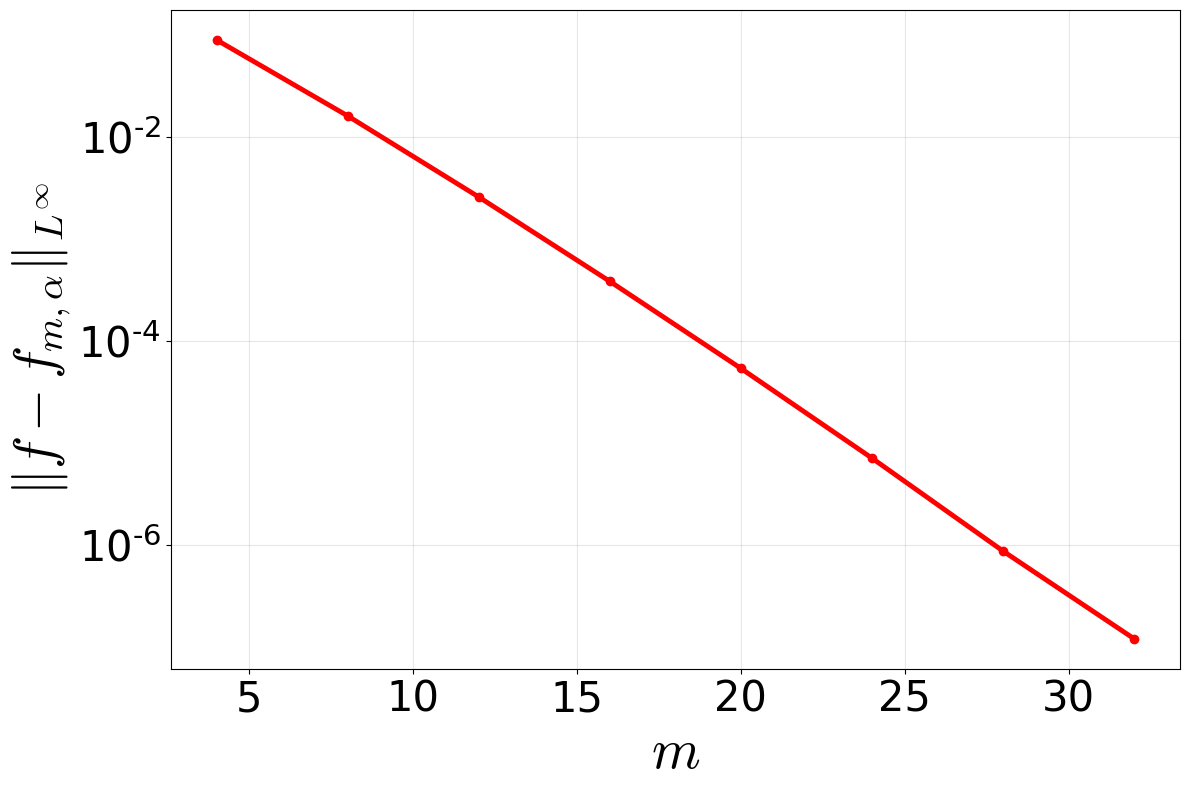}}
\caption{$L^{\infty}$ error decay of the fractional IPRM for six test functions with $\alpha=\pi/4$ and $\lambda=0.75$. The horizontal axis is the Gegenbauer truncation degree $m$. All functions exhibit exponential convergence, with rates determined by the size of the analyticity region $\rho$.}
\label{fig:error_decay}
\end{figure}

The parameter $\rho$ is determined \cite[(8.4)--(8.5)]{trefethen2019} by the nearest singularity of $f$ in the complex plane after mapping each subinterval to $[-1,1]$. Specifically, consider $f_1$ on the subinterval $[-1,0]$: the affine map $t=2x+1$ sends $[-1,0]$ to $[-1,1]$, and the nearest pole $x=i/5$ of $1/(1+25x^2)$ maps to $z_0=1+2i/5$. Then, the Bernstein ellipse parameter is
\[
\rho=|z_0+\sqrt{z_0^2-1}|\approx 1.92.
\]

For a piecewise function, $\rho$ is taken as the minimum over all subintervals, since the slowest-converging piece determines the overall rate. Table \ref{tab:error_decay} lists the $\rho$ values alongside the $L^{\infty}$ errors for each function at selected values of $m$. The error decreases by a nearly constant factor as $m$ increases by a fixed increment, confirming the exponential rate. For instance, the error of $f_5$ decreases by a factor of approximately $17$--$18$ per $\Delta m=4$ in the range $m=16$ to $m=28$, which is consistent with the theoretical prediction $\rho^4\approx 2.06^4\approx 18.0$. The function $f_2$ achieves an error below $10^{-10}$ at $m=32$ owing to its large $\rho\approx 2.15$, while $f_6$ reaches only $10^{-7}$ at the same $m$ due to its smaller $\rho\approx 1.62$.

\begin{table}[htbp]
\caption{Absolute $L^{\infty}$ errors of the fractional IPRM for varying $m$, with $\alpha=\pi/4$, $\lambda=0.75$, and $N=10m$. The Bernstein ellipse parameter $\rho$ 
(minimum over all subintervals) characterizes the theoretical convergence rate.}\label{tab:error_decay}
\centering
\setlength{\tabcolsep}{4mm}{
\renewcommand{\arraystretch}{1.2}
\begin{tabular}{ccccccc}
\toprule
$m$ & $f_1$ & $f_2$ & $f_3$ & $f_4$ & $f_5$ & $f_6$ \\
\midrule
 4  & 1.01e-1 & 6.08e-2 & 4.83e-2 & 1.22e-1 & 7.91e-2 & 8.87e-2 \\
 8  & 4.50e-3 & 2.83e-3 & 4.88e-3 & 2.51e-2 & 5.90e-3 & 1.60e-2 \\
12  & 3.08e-4 & 1.72e-4 & 4.40e-4 & 1.27e-3 & 3.89e-4 & 2.57e-3 \\
16  & 6.29e-5 & 8.75e-6 & 3.68e-5 & 1.71e-4 & 2.39e-5 & 3.83e-4 \\
20  & 5.80e-6 & 3.21e-7 & 2.90e-6 & 3.55e-5 & 1.39e-6 & 5.37e-5 \\
24  & 2.90e-7 & 2.55e-8 & 2.15e-7 & 2.23e-6 & 7.75e-8 & 7.10e-6 \\
28  & 4.21e-8 & 3.95e-10& 1.49e-8 & 1.17e-7 & 4.16e-9 & 8.73e-7 \\
32  & 1.58e-9 & 9.14e-11& 1.16e-9 & 3.45e-8 & 2.58e-10& 1.21e-7 \\
\midrule
$\rho$ & 1.92 & 2.15 & 1.87 & 1.78 & 2.06 & 1.62\\
\bottomrule
\end{tabular}}
\end{table}

Although the convergence rate depends on $\rho$ (and hence on the analyticity of $f$), the error estimate in Theorem \ref{thm:error} and the conditioning analysis in Theorem \ref{thm:condW} both predict that the reconstruction accuracy should be 
insensitive to the transform angle $\alpha$, since the underlying Gram matrix $\mathrm{Gr}$ is independent of $\alpha$ (Proposition~\ref{prop:chirp}). To verify 
this prediction, Figure \ref{fig:error_alpha} plots the relative deviation
\[
\frac{|e_\alpha(m)-\bar{e}(m)|}{\bar{e}(m)},
\]
where $e_\alpha(m)=\|f-f_{m,\alpha}\|_{L^{\infty}}$ is the $L^{\infty}$ reconstruction error at angle $\alpha$, and $\bar{e}(m)=\frac{1}{|\mathcal{A}|}
\sum_{\alpha\in\mathcal{A}}e_\alpha(m)$ is the mean error over the set $\mathcal{A}=\{\pi/16,\pi/8,3\pi/16,\pi/4,5\pi/16,3\pi/8,7\pi/16\}$ of seven tested angles. For all six functions, the relative deviation remains below $10^{-3}$ throughout the range of $m$, and for most functions stays at the level of $10^{-5}$--$10^{-6}$, 
confirming that the choice of $\alpha$ has negligible effect on the reconstruction accuracy.

\begin{figure}[htbp]
\centering
\subfloat[Piecewise function: $f_1(x)$]{\label{fig:error_alpha_f1}\includegraphics[width=0.32\textwidth]{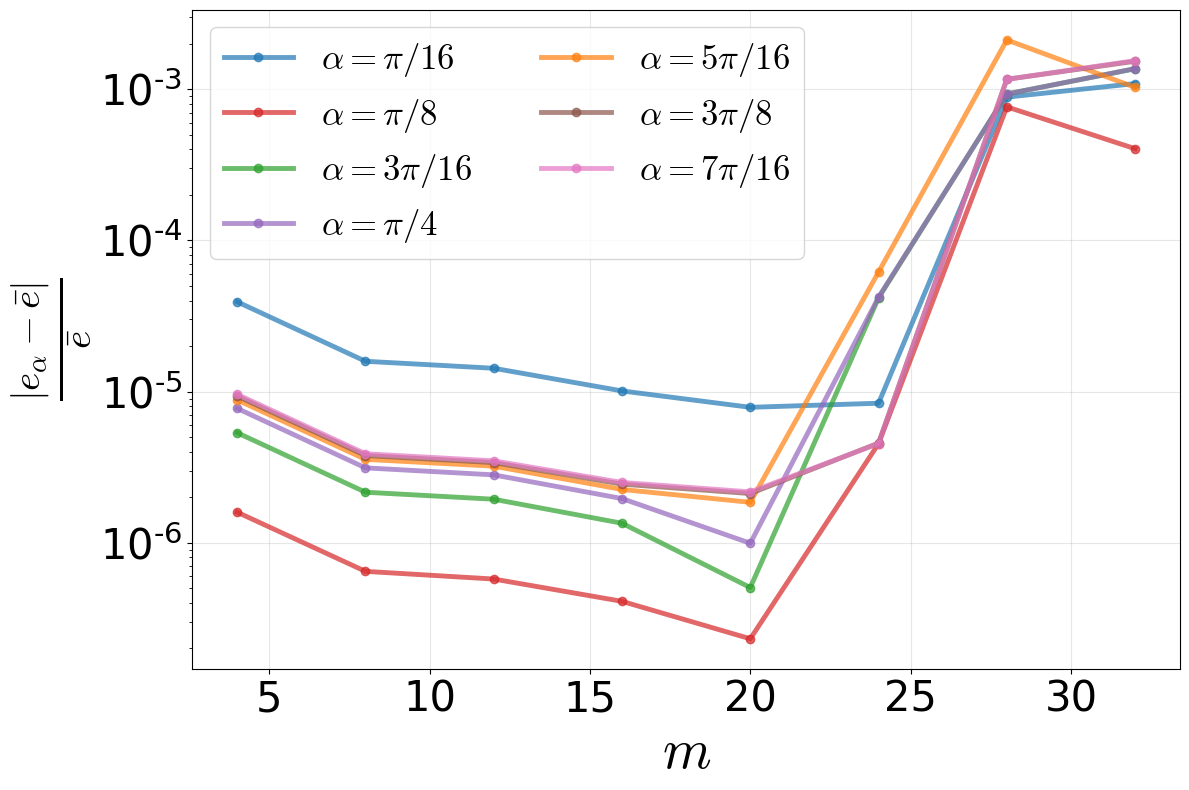}}
\hspace{1mm}
\subfloat[Asymmetric function: $f_2(x)$]{\label{fig:error_alpha_f2}\includegraphics[width=0.32\textwidth]{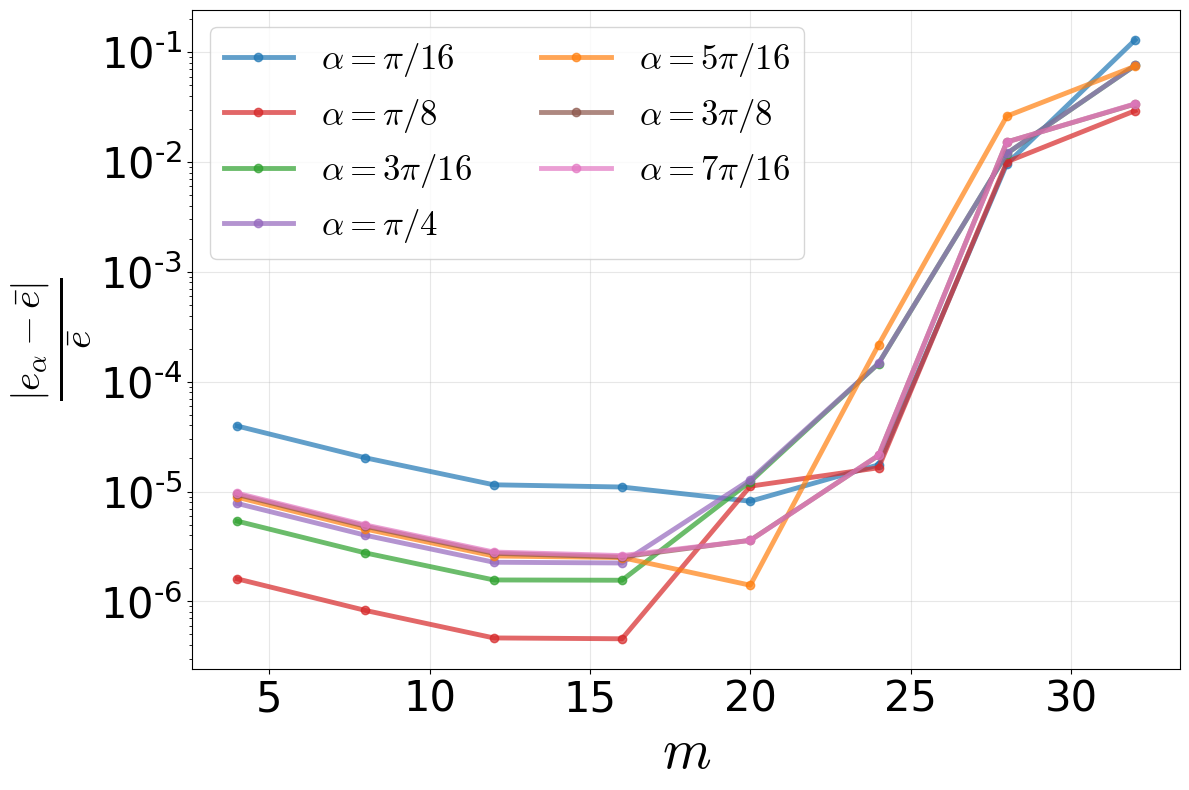}}
\hspace{1mm}
\subfloat[Two-jump function: $f_3(x)$]{\label{fig:error_alpha_f3}\includegraphics[width=0.32\textwidth]{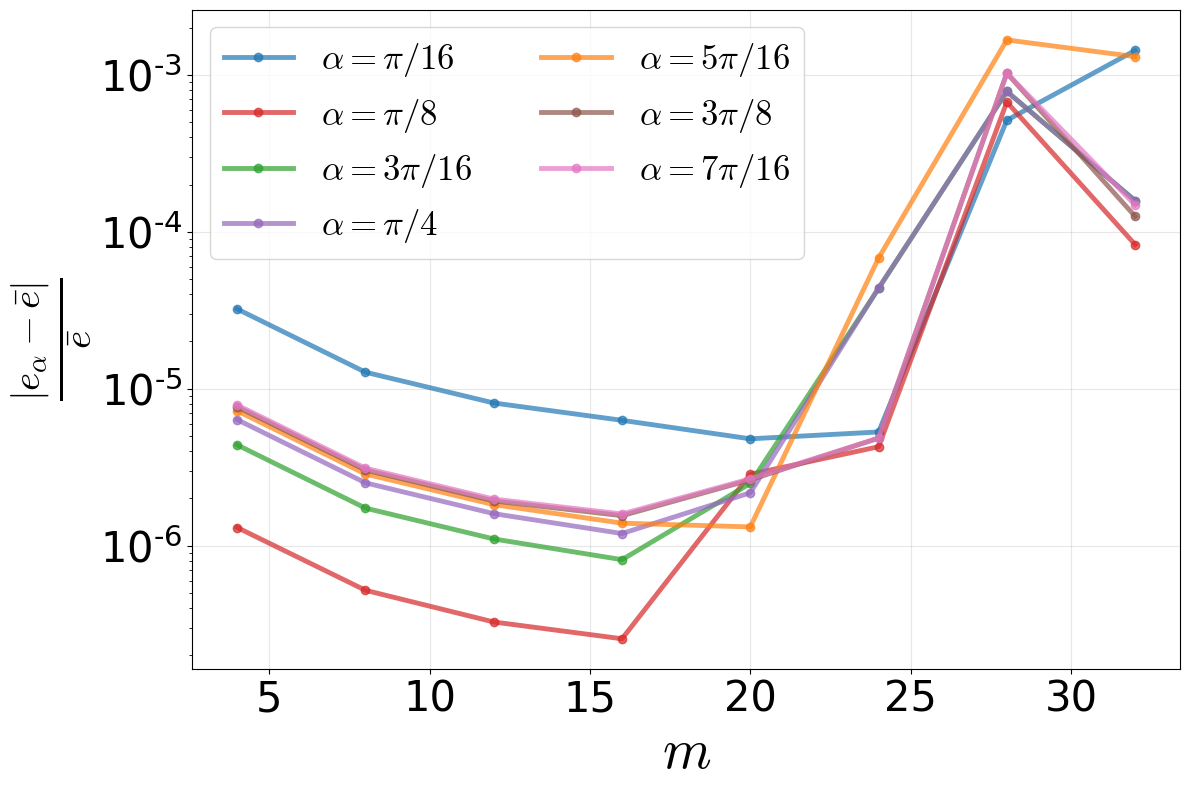}}
\hspace{1mm}
\subfloat[Piecewise tanh: $f_4(x)$]{\label{fig:error_alpha_f4}\includegraphics[width=0.32\textwidth]{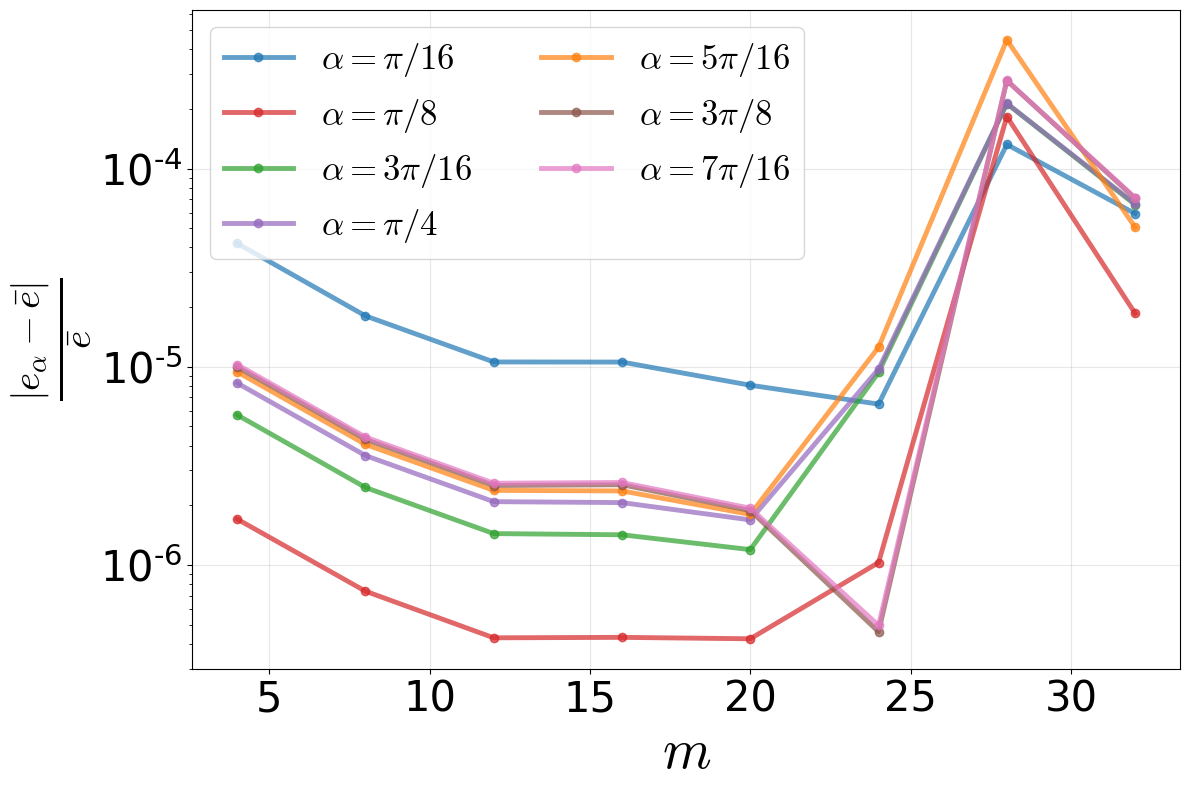}}
\hspace{1mm}
\subfloat[Two-jump tanh: $f_5(x)$]{\label{fig:error_alpha_f5}\includegraphics[width=0.32\textwidth]{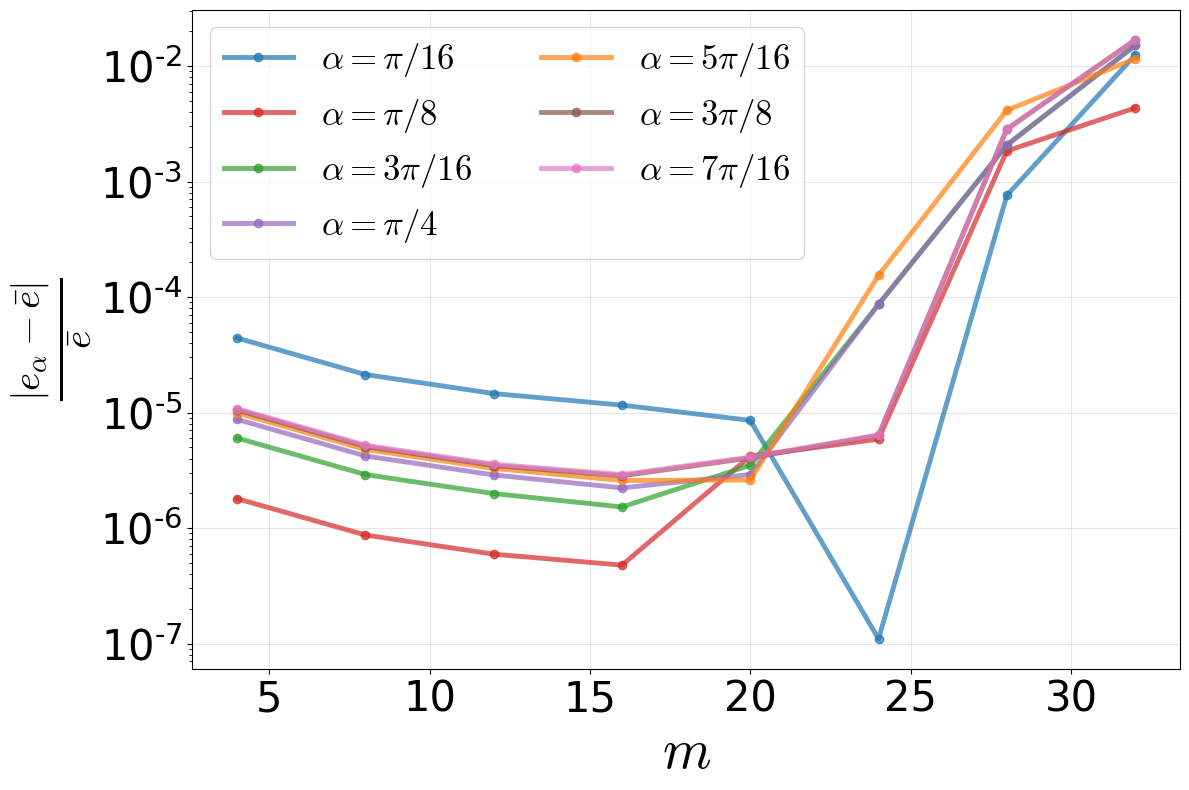}}
\hspace{1mm}
\subfloat[Three pieces: $f_6(x)$]{\label{fig:error_alpha_f6}\includegraphics[width=0.32\textwidth]{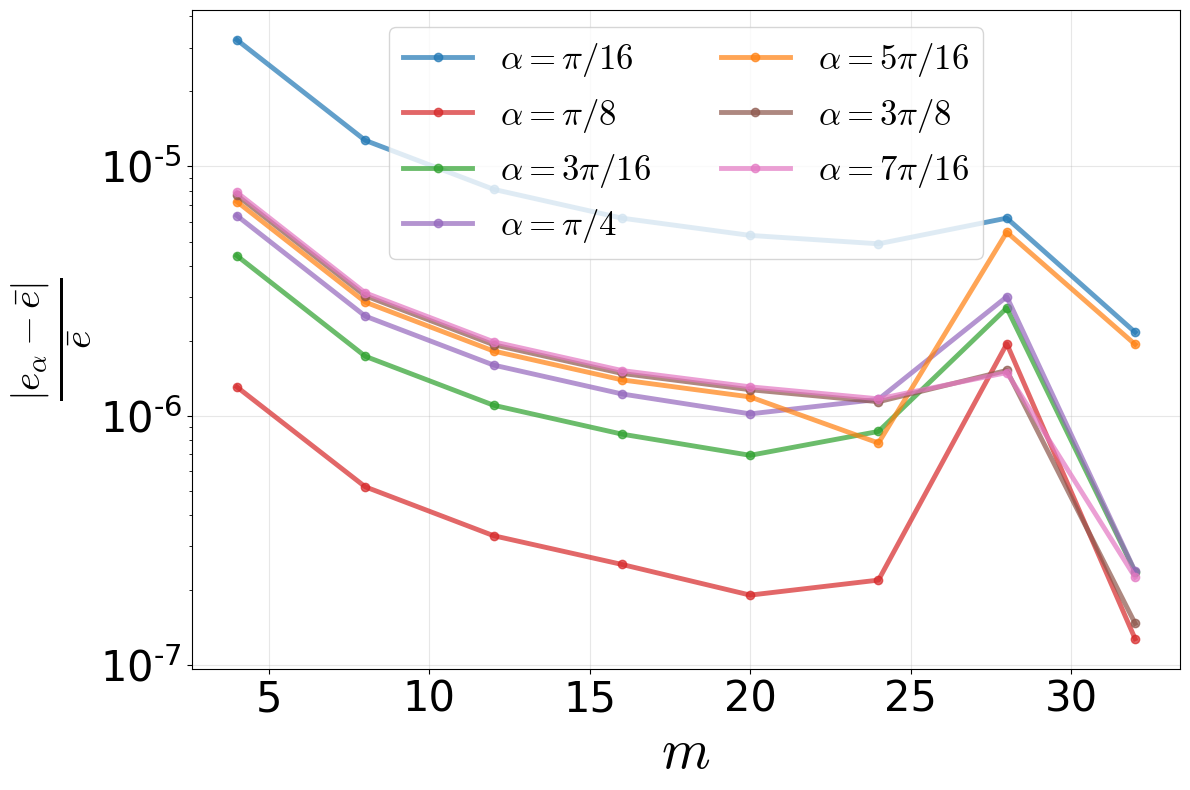}}
\caption{Insensitivity of the reconstruction error to the transform angle $\alpha$. The vertical axis shows the relative deviation $|e_\alpha(m)-\bar{e}(m)|/\bar{e}(m)$ of the $L^{\infty}$ error at each $\alpha$ from the mean error $\bar{e}(m)$ over seven angles from $\pi/16$ to $7\pi/16$, with $\lambda=0.75$ and $N=10m$.}
\label{fig:error_alpha}
\end{figure}

In summary, the numerical experiments in this section confirm the two main theoretical predictions of Sections \ref{subsec:solvability} and \ref{subsec:error}: the fractional IPRM achieves exponential error decay at a rate governed by the 
Bernstein ellipse parameter $\rho$ of the target function, and the reconstruction accuracy is independent of the transform angle $\alpha$. These properties, combined with the complete elimination of the Gibbs phenomenon demonstrated in Section \ref{subsec:reconstruction}, establish the fractional IPRM as a robust and accurate method for reconstructing piecewise analytic functions from fractional Fourier data.

\section{Conclusion}\label{sec:conclusion}

We have presented the first extension of the Inverse Polynomial Reconstruction Method to one-dimensional fractional Fourier series. The method reconstructs 
piecewise analytic functions by solving an overdetermined system that maps fractional Fourier coefficients to Gegenbauer expansion coefficients, bypassing the non-commutativity of projections that limits the direct Gegenbauer method. The conditioning analysis reveals that the coefficient matrix $W_\alpha$ is governed by an $\alpha$-independent Gram matrix $\mathrm{Gr}$, whose eigenvalue bounds identify $\lambda\approx 0.75$ as the optimal Gegenbauer parameter. An $L^{\infty}$ error estimate establishes exponential convergence for analytic functions at a rate determined by the Bernstein ellipse parameter $\rho$. Numerical experiments on 
six piecewise analytic test functions confirm that the method completely eliminates the Gibbs phenomenon, achieves errors three to five orders of magnitude smaller than both the fractional Fourier partial sum and the direct Gegenbauer method, and is insensitive to the transform angle $\alpha$.

In future work, we plan to extend the proposed method to multi-dimensional fractional Fourier series and to investigate stable reconstruction strategies in the presence of noisy fractional Fourier data.

\bibliographystyle{elsarticle-num}
\bibliography{ref_IPRM}

\end{document}